\documentclass{compositio}

\usepackage{amsmath,amsthm,amssymb,mathdots,enumerate,color,soul}

\usepackage[all,cmtip]{xy}
\usepackage{xcolor}
\usepackage[backref=page]{hyperref}
\definecolor{couleur_cite}{rgb}{0.05,.4,0.05}
\definecolor{couleur_link}{rgb}{0.05,0.05,0.4}
\hypersetup{
bookmarksopen,bookmarksnumbered,
colorlinks,
linkcolor=couleur_link,citecolor=couleur_cite,
}

\newtheorem{theorem}{Theorem}[section]
\newtheorem{lemma}[theorem]{Lemma}
\newtheorem{prop}[theorem]{Proposition}

\newtheorem{cor}[theorem]{Corollary}
\newtheorem{definition}[theorem]{Definition}

\newtheorem{conj}{Conjecture}

\theoremstyle{remark}
\newtheorem{remark}[theorem]{Remark}

\newcommand{\R}{\mathbb R}
\newcommand{\C}{\mathbb C}
\newcommand{\N}{\mathbb N}
\newcommand{\Z}{\mathbb Z}
\newcommand{\Q}{\mathbb Q}

\newcommand{\A}{\mathbb A}

\newcommand{\p}{\mathfrak p}
\newcommand{\g}{\mathfrak g}

\newcommand{\gk}{\mathfrak k}

\newcommand{\ga}{\mathfrak a}

\newcommand{\gn}{\mathfrak n}
\newcommand{\gh}{\mathfrak h}

\newcommand{\cO}{\mathcal O}

\newcommand{\cL}{\mathcal L}

\newcommand{\cD}{\mathcal D}

\newcommand{\cP}{\mathcal P}

\newcommand{\cH}{\mathcal H}

\newcommand{\cC}{\mathcal C}
\newcommand{\cN}{\mathcal N}

\newcommand{\SL}{\text{SL} }
\newcommand{\GL}{\text{GL} }

\newcommand{\SU}{\text{SU} }
\newcommand{\SO}{\text{SO} }
\newcommand{\Sp}{\text{Sp} }
\newcommand{\U}{\text{U} }

\newcommand{\ad}{\textup{ad}}
\newcommand{\Ad}{\textup{Ad}}

\newcommand{\tr}{\text{tr}}

\newcommand{\be}{\begin{equation}}
\newcommand{\ee}{\end{equation}}
\newcommand{\bes}{\begin{equation*}}
\newcommand{\ees}{\end{equation*}}

\newcommand{\ba}{\begin{eqnarray}}
\newcommand{\ea}{\end{eqnarray}}
\newcommand{\bas}{\begin{eqnarray*}}
\newcommand{\eas}{\end{eqnarray*}}

\begin{document}
\title[Lower bounds for Maass forms on semisimple groups]{Lower bounds for Maass forms on semisimple groups}

\author{Farrell Brumley}
\email{brumley@math.univ-paris13.fr}
\address{Universit\'e Sorbonne Paris Nord\\ 
Laboratoire de G\'eom\'etrie, Analyse et Applications\\ 
LAGA, CNRS, UMR 7539\\ 
F-93430, Villetaneuse, France }
\thanks{Supported by ANR grant 14-CE25}

\author{Simon Marshall}
\email{marshall@math.wisc.edu}
\address{Department of Mathematics\\
University of Wisconsin -- Madison\\
480 Lincoln Drive\\
Madison\\
WI 53706, USA}
\thanks{Supported by NSF grant DMS-1902173.}


\begin{abstract}
Let $G$ be an anisotropic semisimple group over a totally real number field $F$. Suppose that $G$ is compact at all but one infinite place $v_0$. In addition, suppose that $G_{v_0}$ is $\R$-almost simple, not split, and has a Cartan involution defined over $F$. If $Y$ is a congruence arithmetic manifold of non-positive curvature associated with $G$, we prove that there exists a sequence of Laplace eigenfunctions on $Y$ whose sup norms grow like a power of the eigenvalue.
\end{abstract}

\maketitle

\section{Introduction}

Let $Y$ be a closed Riemannian manifold of dimension $n$ and with Laplace operator $\Delta$.  Let $\{ \psi_i \}$ be an orthonormal basis of Laplace eigenfunctions for $L^2(Y)$, which satisfy $\| \psi_i \|_2 = 1$ and $(\Delta + \lambda_i^2) \psi_i = 0$.  We assume that $\{ \psi_i \}$ are ordered by eigenvalue, so that $0 = \lambda_1 \leqslant \lambda_2 \leqslant \ldots$.  It is an important question in harmonic analysis to determine the asymptotic size of $\psi_i$, i.e. the growth rate of $\| \psi_i \|_\infty$ in terms of $\lambda_i$.  The basic upper bound for $\| \psi_i \|_\infty$, proved by Avacumovi\'c \cite{Av} and Levitan \cite{Le}, is given by
\be
\label{AL}
\| \psi_i \|_\infty \ll \lambda_i^{(n-1)/2}.
\ee
This bound is sharp on the round $n$-sphere. Indeed, the zonal spherical harmonics have peaks of maximal size at the poles of the axis of rotation. More generally, Sogge and Zelditch \cite{SZ} have shown that the compact Riemannian manifolds saturating \eqref{AL} necessarily have points which are fixed by an appropriately large number of geodesic returns, in the sense that a positive measure subset of geodesics passing through such a point are loops.

On the other hand, if $Y$ is negatively curved then its geodesic flow is highly chaotic, and one expects this to be reflected in the asymptotics of the eigenfunctions. For example, the quantum ergodicity theorem of Schnirelman \cite{Schnirelman}, Colin de Verdi\`ere \cite{CdV}, and Zelditch \cite{Zelditch} states that the $L^2$-mass of a density one sequence of Laplacian eigenfunctions on a negatively curved manifold equidistributes to the uniform measure. It is likewise reasonable to expect that if $Y$ has negative curvature then the strong pointwise delocalization bound
\be\label{Lindeloff}
\| \psi_i \|_\infty \ll_\epsilon \lambda_i^\epsilon
\ee
holds with density one. This is akin to the Ramanujan conjecture in the theory of automorphic forms \cite{Sar1993}: a generic sequence of eigenfunctions is tempered. Any sequence violating \eqref{Lindeloff} will be called {\it exceptional}.

Unlike the weak-* setting of the Quantum Unique Ergodicity conjecture \cite{RS}, which posits that in negative curvature the $L^2$-mass of {\it any} sequence of eigenfunctions equidistributes to the uniform measure, there do in fact exist compact manifolds $Y$ of negative curvature violating \eqref{Lindeloff}. The first examples were given in \cite{RS}, with $Y$ being an arithmetic hyperbolic $3$-manifold and the exceptional sequences, of density zero, lying in the image of the theta corresponence from $\Sp_2$. 

For surfaces, the situation is quite different: Iwaniec and Sarnak \cite{IS} conjecture that hyperbolic surfaces do not support exceptional sequences. This is a very difficult problem which would imply the classical Lindel\"of conjecture on the Riemann zeta function in the case of the modular surface. 

This leads to the question of when, exactly, should one expect to find (zero density) violations to \eqref{Lindeloff}. In this paper we give sufficient conditions for a certain class of negatively curved manifolds to support exceptional sequences. Although the question is of interest in this general setting, our techniques are limited to arithmetic locally symmetric spaces. Put succinctly, we show that an arithmetic manifold supports exceptional sequences whenever it has a point with strong Hecke return properties.

\subsection{Statement of results}
\label{sec:statement}

Our main theorem is modelled on a result of Mili\'cevi\'c \cite{Mi}, which, building on \cite{RS}, provides a structural framework for the class of arithmetic hyperbolic $3$-manifolds supporting exceptional sequences.

First recall that an arithmetic hyperbolic $3$-manifold arises from the following general construction. Let $E$ be a number field having exactly one complex embedding, up to equivalence, and let $F$ be its maximal totally real subfield. For a division quaternion algebra $B$ over $E$, ramified at all real places of $E$, denote by $G$ the restriction of scalars of $B^1$ from $E$ to $F$. Then any arithmetic hyperbolic $3$-manifold is commensurable with a congruence manifold associated with $G$.

Following \cite{Mi}, an arithmetic hyperbolic $3$-manifold as above is said to be of Maclachlan--Reid type if $E$ is quadratic over $F$ and there exists a quaternion division algebra $A$ over $F$ satisfying $B=A\otimes_F E$. The main result of {\it loc. cit.} is that Maclachlan--Reid type manifolds support exceptional sequences (in fact, satisfying the same lower bounds as the examples of Rudnick--Sarnak).

Notice that when $B=A\otimes_F E$, the following properties hold. Let $v_0$ be the unique archimedean place of $F$ which ramifies in $E$. By \cite[Theorem 9.5.5]{MacReid} we may assume that $A$ is ramified at $v_0$. Then
\begin{enumerate}
\item $G_{v_0}=\SL_2(\C)$ is noncompact, and non-split (as an $\R$-group);
\item $G_v=\mathbf{H}^1$ (the norm-one Hamiltonian quaternions) is compact for all real $v\neq v_0$;
\item the global involution $\theta: g\mapsto\sigma(g)$ of $G$, where $\sigma$ is the unique non-trivial element in the Galois group of $E$ over $F$, induces a Cartan involution on $G_{v_0}$. Indeed, $G^\theta=A^1$, so that $G^\theta(F_{v_0})=\mathbf{H}^1$ is the maximal compact $\SU(2)$ inside $G_{v_0}=\SL_2(\C)$.
\end{enumerate}

Our main result is an extension of this to a wide range of compact congruence manifolds. 

\begin{theorem}\label{thm1}

Let $F$ be a totally real number field, and let $v_0$ be a real place of $F$.  Let $G/F$ be a connected anisotropic semisimple $F$-group.  We make the following additional assumptions on $G$.

\begin{enumerate}

\item
\label{G1}
$G_{v_0}$ is noncompact, not split, and $\R$-almost simple.

\item
\label{G2}
$G_v$ is compact for all real $v \neq v_0$.

\item
\label{G3}
There is an involution $\theta$ of $G$ defined over $F$ that induces a Cartan involution of $G_{v_0}$.

\end{enumerate}

Let $Y$ be a congruence manifold associated with $G$ as in Section \ref{adelic-quotients}.  Then there exists $\delta > 0$, depending only on $G$, and a sequence of linearly independent Laplacian eigenfunctions $\psi_i$ on $Y$ that satisfy
\[
\| \psi_i \|_2 = 1,\quad (\Delta + \lambda_i^2) \psi_i = 0,\quad\text{ and }\quad \| \psi_i \|_\infty \gg \lambda_i^\delta.
\]

\end{theorem}

We note that the exponent $\delta$ in Theorem \ref{thm1}, and Theorem \ref{thm2} below, is currently logically ineffective. This comes from using a result of Cluckers--Gordon--Halupczok \cite[Theorem 14.1]{ST} to bound orbital integrals. (Note that Shin--Templier present an alternate, effective proof of \cite[Theorem 14.1]{ST} in Theorem 7.3 of \textit{loc. cit.}, but this has an error as described on S.-W. Shin's webpage.)  The bound of Cluckers--Gordon--Halupczok is ineffective due to their use of model theory.  Julia Gordon informs us that, if the residue characteristic is larger than a logically ineffective bound, it may be possible to transfer the constants in their result from the case of positive characteristic and explicate them in that case.  However, this is presently unknown.  Moreover, even if this were possible, the implied constant in Theorems \ref{thm1} and \ref{thm2} would still be logically ineffective because of the condition that the residue characteristic be sufficiently large.

An additional source of ineffectiveness of $\delta$ comes from using theorems of Finis--Lapid \cite{FL} when the congruence manifold varies as in Theorem \ref{thm2}, although this ineffectiveness is in the weaker sense that the constants are not numerically explicit.  Finis and Lapid inform us that their results could probably be made effective with extra work; in fact, as the referee has pointed out, the particular shape of the congruence subgroups, which depend on $\theta$, in Theorem \ref{thm2} should allow for the exponents arising from their work to be more readily explicated.

\subsection{Remarks on the main theorem}
A well-known theorem of Borel \cite{Borel63} addresses the question of whether one can find many groups satisfying the rationality hypothesis \eqref{G3}. One consequence of his theorem is that for any connected, simply-connected, semisimple algebraic $\R$-group $G$ satisfying condition \eqref{G1}, Theorem \ref{thm1} produces a manifold $Y$ of the form $\Gamma \backslash G / K$ with an exceptional sequence of eigenfunctions. See Section \ref{sec-Borel} for more details and a concrete example.

Theorem \ref{thm1} goes some distance toward answering the basic question of determining the precise conditions under which one should expect a Lindel\"of type bound on a compact congruence negatively curved manifold.  The three numbered conditions on the group $G$ are a particularly convenient way of asking that a large enough compact subgroup of $G_\infty$ admits a rational structure, which is a key ingredient in our proof.  Although the condition that $G_{v_0}$ is not split should be necessary, we expect that the other conditions can be relaxed somewhat. For example, throughout most of the paper, the condition that $G_{v_0}$ is $\R$-almost simple could be weakened to $G$ being $F$-almost simple. The stronger form of this condition is only used in Lemma \ref{BPlemma}, to simplify the application of a theorem of Blomer--Pohl \cite[Theorem 2]{BP} and Matz--Templier \cite[Proposition 7.2]{MT}.

Besides the results of Rudnick--Sarnak and Mili\'cevi\'c that we have already mentioned, both in the context of arithmetic hyperbolic $3$-manifolds, there are other results in the literature which provide examples of arithmetic manifolds supporting exceptional sequences. For instance, the techniques of Rudnick--Sarnak were generalised to $n$-dimensional hyperbolic manifolds for $n \geqslant 5$ by Donnelly \cite{Do}. Later, Lapid and Offen in \cite{LO} discovered a series of arithmetic quotients of $\SL(n,\C) / \SU(n)$ admitting large eigenforms through the link with automorphic $L$-functions (conditionally on standard conjectures on the size of automorphic $L$-functions at the edge of the critical strip). To our knowledge \textit{loc. cit.} is the only example of power growth lower bounds for sup norms on non-compact quotients, at least if one restricts to a compact set so that the large peaks of eigenfunctions in the cusp do not play a role. Note that Theorem \ref{thm1} includes the examples of Rudnick--Sarnak, Donnelly, and Mili\'cevi\'c, although without explicit exponents. It is unable to reproduce the examples of Lapid--Offen due to the compactness requirement, but -- as was indicated above -- it can produce compact quotients of $\SL(n,\C) / \SU(n)$ with an exceptional sequence of eigenfunctions. In fact, non-compact quotients (of suitable level structure) should also be amenable to our techniques, via an application of simple trace formulae, but we have not pursued this here. 

Finally, while our approach was largely inspired by that of Mili\'cevi\'c, we have made an effort to emphasize (in Appendix \ref{sym-var}) the common features it shares with the techniques of Rudnick--Sarnak and Lapid--Offen. A synthesis of the subject, as well as a general conjecture restricting the possible limiting exponents for exceptional sequences, can be found in the influential letter \cite{SarMor}.

\subsection{A hybrid result in the level-eigenvalue aspect}

We in fact prove a stronger result than that described in Theorem \ref{thm1}, establishing a lower bound in the level and eigenvalue aspects simultaneously.  We present this separately, as it requires more care to state; indeed, any notion of non-trivial lower bound must overcome the lower bound one may prove when the eigenspaces have large dimension.  More precisely, if $M$ is a compact Riemannian manifold and $V$ is the space of $\psi \in L^2(M)$ with a given Laplace eigenvalue, one may show that there is $\psi \in V$ satisfying $\| \psi \|_\infty \geqslant \sqrt{\dim V} \| \psi \|_2$.

If we consider a tower of congruence covers $Y_N$ of $Y$, then the Laplace eigenspaces will have growing dimension because of multiplicities in the corresponding representations at places dividing $N$.  Computationally, one observes that this (and its stronger form involving Arthur packets) is the only source of dimension growth.  Although we believe that the dimensions of the joint eigenspaces we consider should be small (partly as a result of our choice of ``large'' congruence subgroup), we do not know how to prove this in general. As a result, we shall be satisfied if we can beat the bound $\sqrt{\dim V}$, where $V$ is now a space of Hecke--Maass forms with the same Laplace and Hecke eigenvalues.  This motivates the following definitions.

Let $G$ be as in Theorem \ref{thm1}. Let $H$ be the identity component of the group of fixed points of $\theta$. We let $D$ be a positive integer such that $G$ and $H$ are unramified at places away from $D$ and $\infty$; see Section \ref{sec:AG} for a precise definition. Let $K$ and $K_H$ be compact open subgroups of $G(\A_f)$ and $H(\A_f)$ that are hyperspecial away from $D$, as in Sections \ref{cmp-sg} and \ref{adelic-quotients}. If $N$ is a positive integer prime to $D$, we let $K(N)$ be the corresponding principal congruence subgroup of $K$, and define $Y_N = G(F) \backslash G(\A) / K(N) K_H K_\infty$.  We give each $Y_N$ the probability volume measure.

Let $A \subset G_\infty$ be a maximal $\R$-split torus with real Lie algebra $\ga$ and Weyl group $W$.  We let $\ga_\C = \ga \otimes \C$.  Let $G_\infty^0$ be the connected component of $G_\infty$ in the real topology.  Any unramified irreducible unitary representation of $G_\infty^0$ gives rise to an element $\xi \in \ga_\C^* / W$ via the Harish-Chandra isomorphism, which we have normalised so that the tempered spectrum corresponds to $\ga^* / W$.  We let $\| \cdot \|$ be the norm on $\ga$ and $\ga^*$ coming from the Killing form and extend it naturally to their complexifications. If $\mu, \lambda \in \ga_\C^* / W$ we will sometimes abuse notation and write $\| \mu - \lambda \|$ to mean the minimum of this norm over representatives for the $W$-orbits.

By a Hecke--Maass form we mean a joint eigenfunction $\psi\in L^2(Y_N)$ for the Hecke algebra (away from $N$ and $D$) and the ring of invariant differential operators $\mathcal{D}$ on $Y_N$.  We may view the associated eigenvalues as elements in the unramified unitary dual of $G_v$ at finite places $v$ (via the Satake isomorphism), while at infinity they determine an element $\xi \in \ga_\C^* / W$. We define a spectral datum $c$ for $(G,N)$ to be a choice of element $\xi(c) \in \ga_{\C}^* / W$ and an element $\pi_v(c)$ in the unramified unitary dual of $G_v$ for all $v \nmid ND\infty$.  Given a spectral datum $c$ for $(G,N)$, we define $V(N,c)$ to be the space of Hecke--Maass forms on $Y_N$ whose $\mathcal{D}$-eigenvalues are given by $\xi(c)$ (the spectral parameter) and whose Hecke eigenvalues at $v \nmid ND \infty$ are given by $\pi_v(c)$.

\begin{theorem}\label{thm2}

With the notation and hypotheses of Theorem \ref{thm1}, there is $\delta>0$, depending only on $G$ and $\theta$, and $Q>1$ with the following property. For any positive integer $N$ with $(N, D) = 1$ and spectral parameter $\xi \in \ga^*$ such that $N (1+\|\xi\|)$ is sufficiently large, there is a spectral datum $c$ for $(G,N)$ with $\|\xi(c)-\xi\|\leqslant Q$ and a Hecke--Maass form $\psi\in V(N,c)$ such that 
\[
\| \psi \|_\infty \gg N^\delta  (1+\|\xi\|)^\delta \sqrt{ \dim V(N,c)}\|\psi\|_2.
\]

\end{theorem}

Note that a Hecke--Maass form as in Theorem \ref{thm1} has Laplacian eigenvalue of size roughly $(1+\|\xi\|)^2$. Theorem \ref{thm2} therefore implies Theorem \ref{thm1}.  Moreover, taking $\xi$ at distance at least $Q$ from the root hyperplanes ensures that the eigenfunction $\psi$ produced by the theorem is tempered at infinity.  As in Theorem \ref{thm1}, the exponent $\delta$ is ineffective.

The only previous results giving lower bounds in the level aspect are for $\GL_2$ over a number field, due to Saha \cite{Sh} and Templier \cite{T}. They use the fact that local Whittaker functions of highly ramified $p$-adic representations are large high in the cusp, and in particular rely on the noncompactness of the manifold.

\subsection{Borel's theorem and a concrete example}\label{sec-Borel}

A classical theorem of Borel on the existence of rationally defined Cartan involutions on real semisimple Lie algebras can be used to provide examples of groups satisfying the hypotheses of Theorem \ref{thm1}.\footnote{It is an interesting question whether condition \eqref{G3} on the existence of a rational Cartan involution is automatic or not.  We believe that it is not when $G$ is almost simple of type $A_n$, $D_n$, or $E_6$, but are unsure otherwise.} We state this as the following result and provide details for how to extract this statement from the work of Borel in Section \ref{sec:Borel}.

\begin{prop}
\label{admissibleexist}

Let $G'/\R$ be connected, simply connected, and $\R$-almost simple. Let $F$ be a totally real number field, and let $v_0$ be a real place of $F$. There is a connected semisimple group $G/F$ with $G_{v_0} \simeq G'$ that satisfies conditions \eqref{G2} and \eqref{G3} of Theorem \ref{thm1}.

\end{prop}

We now give a concrete example of a family of manifolds to which our theorem can be applied, and which to our knowledge does not already appear in the literature.

Let $F$ be a totally real number field, and let $E$ be a CM extension of $F$. Let the rings of integers of these fields be $\cO_F$ and $\cO_E$ respectively. Let $v_0$ be a distinguished real place of $F$, and let $w_0$ be the place of $E$ over $v_0$. Let $V$ be a vector space of dimension $n+1$ over $E$ with a Hermitian form $\langle \cdot, \cdot \rangle$ with respect to $E/F$. Assume that $\langle \cdot, \cdot \rangle$ has signature $(n,1)$ at $w_0$ and is definite at all other infinite places of $E$.  Let $G$ be the $F$-algebraic group $\SU(V, \langle \cdot, \cdot \rangle)$, so that $G_{v_0} \simeq \SU(n,1)$.

Let $L \subset V$ be an $\cO_E$ lattice on which the form $\langle \cdot, \cdot \rangle$ is integral. Let $L^*$ be the dual lattice $L^* = \{ x \in V : \langle x, y \rangle \in \cO_E \text{ for all } y \in L \}$. Let $\Gamma$ be the group of isometries of $V$ that have determinant 1, preserve $L$, and act trivially on $L^* / L$. If $F \neq \Q$, completion at $w_0$ allows us to consider $\Gamma$ as a discrete, cocompact subgroup of $\SU(n,1)$, which will be torsion free if $L$ is chosen sufficiently small.

One may associate a complex hyperbolic manifold with $\Gamma$ in the following way. Let $D$ denote the space of lines in $V_{w_0}$ on which the Hermitian form is negative definite. $\SU(n,1)$ acts on $D$, and $D$ carries a natural $\SU(n,1)$-invariant metric under which it becomes a model for complex hyperbolic $n$-space. The quotient $Y = \Gamma \backslash D$ is then a compact complex hyperbolic $n$-manifold, and is an example of a congruence manifold associated with $G$ as in Theorem \ref{thm1}.

If $n \geqslant 2$, $G$ satisfies conditions \eqref{G1} and \eqref{G2} of Theorem \ref{thm1}. We now show that \eqref{G3} is satisfied. Let $W \subset V$ be a codimension 1 subspace defined over $E$ such that the Hermitian form is positive definite on $W_{w_0}$. Let $\theta$ be the isometry of reflection in $W$. Then $g \mapsto \theta g \theta^{-1}$ gives an $F$-involution of $G$ that is a Cartan involution on $G_{v_0}$, as required. Theorem \ref{thm1} then implies that there is a sequence of Laplace eigenfunctions $\{ \psi_i \}$ on $Y$ satisfying $\| \psi_i \|_\infty \gg \lambda_i^\delta \| \psi_i \|_2$. 

\subsection{The method of proof}\label{intro:idea}

The proof of power growth for arithmetic hyperbolic $3$-manifolds of Maclachlan--Reid type by Mili\'cevi\'c \cite{Mi} compares an amplified trace and pre-trace formula. Our proof works by extending this method to general groups. The bulk of the work lies in proving asymptotics for the trace formula.  

More precisely, our proof of Theorem \ref{thm2} proceeds by comparing a trace formula on $G$ with a relative trace formula for $G$ with respect to $H$. If we choose a test function $k \in C^\infty_c(G(\A))$, the main geometric terms of these trace formulae are $k(1)$ and $\Pi_H k(1)$ respectively, where $\Pi_H : L^1(G(\A)) \rightarrow L^1(X(\A))$ is given by integration over $H$. In fact, we need to choose test functions of the form $k=\omega * \omega^*$ so that their action on $L^2(G(F) \backslash G(\A))$ is positive semidefinite. After controlling the other geometric terms, we wish to find $\omega$ that makes $\Pi_H k(1)$ large relative to $k(1)$. At finite places, we shall take for $\omega_f$ an appropriately large sum of $L^2$-normalized basic Hecke operators $\tau(\nu,v)$, supported on $K_v\nu(\varpi_v)K_v$, where $\nu$ is a cocharacter of $G$. The condition that $\Pi_Hk(1)$ be large then boils down to 
\[
\left(\Pi_H\tau(\nu,v)\right)(1)=\frac{{\rm vol}(H_v\cap K_v\nu(\varpi_v)K_v)}{{\rm vol}(K_v\nu(\varpi_v)K_v)^{1/2}}
\]
being large, for enough places $v$. Note that this corresponds to our informal description in terms of Hecke returns at the beginning of the introduction: if the projection of the $H$-period onto the given locally symmetric $Y$ is simply a point $p$ (rather than a finite collection of such), the right-hand side above is roughly the multiplicity with which $p$ appears in its image by the Hecke correspondence $\tau(\nu,v)$. In any case, we bound this quantity from below in Lemma \ref{nu-choice}, in terms of the $H$-relative size of $\nu$. The latter is a certain cocharacter inequality, which we show in Lemma \ref{link2thm} is verified under the conditions of our main theorem.

While writing this paper, Erez Lapid pointed out to us that there was another approach to proving Theorems \ref{thm1} and \ref{thm2} based on a theorem of Sakellaridis on the unramified $C^\infty$ spectrum of symmetric varieties.  We have included a discussion of this in Appendix \ref{sym-var}.  We have also included an explanation of why the condition of $G$ being nonsplit at $v_0$ is natural, and motivated our choice of test functions in the trace formula, based on a related conjecture of Sakellaridis and Venkatesh on the $L^2$ spectrum.

\subsection*{Acknowledgements} We would like to thank Tobias Finis, Erez Lapid, Zeev Rudnick, Yiannis Sakellaridis, Sug Woo Shin, Nicolas Templier, and Akshay Venkatesh for helpful conversations during the elaboration of this paper. The first author benefited from many enlightening discussions with Nicolas Bergeron and Djordje Mili\'cevi\'c.  We would also like to thank the referees for suggestions that greatly improved the paper.

\section{Notation and proof of Proposition \ref{admissibleexist}}\label{sec:not}

\subsection{Algebraic groups}\label{sec:AG}

Let $\cO$ denote the ring of integers of $F$.  Let $\A$ and $\A_f$ be the adeles and finite adeles of $F$.  

Let $G$ and $\theta$ be as in Theorems \ref{thm1} and \ref{thm2}.  Let $H$ again denote the identity component of the group of fixed points of $\theta$.  We let $T \subset G$ and $T_H \subset H$ be maximal tori defined over $F$ with $T_H \subset T$.  We fix an $F$-embedding $\rho : G \rightarrow \SL_d$.  Let $Z$ be the center of $G$.

Consider $\SL_d$ as a group scheme over $\cO$ in the standard way, and take the schematic closures of $T_H$, $T$, $H$, and $G$ in $\SL_d/\cO$.  These closures are group schemes over $\cO$, and all the inclusions between them over $F$ extend to closed embeddings over $\cO$.  By \cite[3.9]{Ti}, we may choose $D \in \N$ such that over $\cO[1/D]$, all fibers of these group schemes are smooth, connected, and reductive.

Let $X^*(T)$ and $X_*(T)$ denote the group of characters and cocharacters of $T \times_F \overline{F}$.  Let $\Delta$ be the set of roots of $T$ in $G$, and let $\Delta^+$ be a choice of positive roots.  Let $W$ be the Weyl group of $(G,T)$ over $\overline{F}$.  We define
\bes
X_*^+(T) = \{ \mu \in X_*(T) : \langle \mu, \alpha \rangle \geqslant 0, \alpha \in \Delta^+ \}.
\ees
Similarly, we may define $\Delta_H,\Delta_H^+$, $W_H$, and $X_*^+(T_H)$. Letting $\rho$ and $\rho_H$ denote, as usual, the half-sum of positive roots for $G$ and $H$, we introduce the functions
\begin{align*}
\| \mu \|^* & = \underset{w \in W}{\max} \langle w \mu, \rho \rangle \\
\| \mu \|^*_H & = \underset{w \in W_H}{\max} \langle w \mu, \rho_H \rangle 
\end{align*}
on $X_*(T)$ and $X_*(T_H)$ respectively.  Our assumption that $G$ is semisimple implies that $\| \mu \|^*$ is a norm; the condition that $\| \mu \|^* = \| -\mu \|^*$ follows from the fact that $\rho$ and $-\rho$ lie in the same Weyl orbit. While $\| \mu \|^*_H$ also has the property that $\| \mu \|^*_H = \| -\mu \|^*_H$, it is trivial on cocharacters arising from the center of $H$.

\begin{definition}
\label{large}

We say that $G$ is $H$-large if there exists a nonzero $\mu \in X_*(T_H)$ such that
\be
\label{K-large-cond}
2 \| \mu \|_H^*  \geqslant \| \mu \|^*.
\ee

\end{definition}

Note that this definition is exactly the opposite of the definition of $H$ being small in $G$ given in \cite[Definition 8.1]{Ma}.  By \cite[Lemma 8.2]{Ma}, the definition of $H$-large is independent of the choice of tori $T_H \subset T$.

\begin{lemma}\label{link2thm}
A group $G$ satisfying the hypotheses of Theorem \ref{thm1} is $H$-large. 
\end{lemma}

\begin{proof}
This follows from \cite[Theorem 8.3]{Ma}. Indeed, \cite[p. 277]{Bo} implies that the compact subgroup $H_{v_0}$ is connected in the real topology; by passing to the Lie algebra, this implies that $H_{v_0}$ is a maximal compact connected subgroup of $G_{v_0}$.  We then deduce from \cite[Theorem 8.3]{Ma} that $H_{v_0}$ is not small in $G_{v_0}$, and invariance of smallness under field extension then implies that $G$ is $H$-large.
\end{proof}

\subsection{Local fields}\label{sec:lf}

If $v$ is a place of $F$, we denote the completion by $F_v$.  If $v$ is finite, we denote the ring of integers, uniformiser, and cardinality of the residue field by $\cO_v$, $\varpi_v$, and $q_v$ respectively.  If $v \nmid D\infty$, we have the following consequences of our assumptions on $D$ and $\rho$ above.

\begin{enumerate}

\item
\label{local1} We have $G(\cO_v) = \rho^{-1}( \rho(G_v) \cap \SL_d(\cO_v))$ and $H(\cO_v) = \rho^{-1}( \rho(H_v) \cap \SL_d(\cO_v))$, so that $G(\cO_v) \cap H_v = H(\cO_v)$.

\item
\label{local2} $G(\cO_v)$ and $H(\cO_v)$ are hyperspecial maximal compact subgroups of $G_v$ and $H_v$ respectively.

\item
\label{local3} If $T$ (and hence $T_H$) split at $v$, the subgroups $G(\cO_v)$ and $H(\cO_v)$ correspond to points in the Bruhat--Tits buildings of $G_v$ and $H_v$ that lie in the apartments of $T$ and $T_H$ respectively.

\end{enumerate}
Property \eqref{local1} is a consequence of the inclusions $H \subset G \subset \SL_d$ extending to closed embeddings over $\cO_v$.  Property \eqref{local2} is \cite[3.8]{Ti} or \cite[5.1.40]{BT2}.  Property \eqref{local3} is \cite[5.1.33]{BT2}, once we use \cite[5.1.40]{BT2} to show that the $\cO_v$ group scheme $T/\cO_v$ is the standard one associated with $T_v$.  We let $\cP$ be the set of finite places of $F$ that do not divide $D$ and at which $T$ splits.  If $v \in \cP$, our assumptions imply that $G_v$ has a Cartan decomposition

\bes
G_v = \coprod_{\mu \in X_*^+(T)} G(\cO_v) \mu(\varpi_v) G(\cO_v)
\ees
with respect to $G(\cO_v)$ and $T$.

\subsection{Metrics}\label{sec:metrics}

For any place $v$ of $F$ and $g \in G(F_v)$ let $\| g\|_v$ denote the maximum of the $v$-adic norms of the matrix entries of $\rho(g)$. For $g \in G(\A_f)$, let $\| g \|_f=\prod_{v \nmid \infty}\|g_v\|_v$. Fix a left-invariant Riemannian metric on $G(F_{v_0})$.  Let $d( \cdot, \cdot)$ be the associated distance function. We define $d(x,y) = \infty$ when $x$ and $y$ are in different connected components of $G(F_{v_0})$ with the topology of a real manifold.

\subsection{Compact subgroups}\label{cmp-sg}
We choose compact subgroups $K_v$ of $G_v$ for all $v$ such that
\begin{enumerate}
\item[$\bullet$] $K_{v_0} = H_{v_0}$,
\item[$\bullet$] $K_v = G_v$ for all other real places,
\item[$\bullet$] $\rho(K_v) \subset \SL_d(\cO_v)$ for all finite $v$,
\item[$\bullet$] $K_v = G(\cO_v)$ for finite places $v \nmid D$, and 
\item[$\bullet$] $K = \prod_{v \nmid \infty} K_v$ is open in $G(\A_f)$. 
\end{enumerate}
By the remarks in the proof of Lemma \ref{link2thm}, $K_{v_0}$ is a maximal compact connected subgroup of $G_{v_0}$.  We shall suppose that $K_v$ for $v|D$ is sufficiently small to ensure that the finite group $Z(F)\cap K$ is reduced to $\{e\}$.

\subsection{Measure normalizations}\label{adelic-measures}

For any place $v$ of $F$, let $\mu_{G,v}^{\rm can}$ be the canonical measure on $G(F_v)$ as defined by Gross in \cite[Section 11]{Gr}; we recall this construction in Section \ref{sec:Grossmeasures}. Then for all finite places $v\nmid D$ one has $\mu_{G,v}^{\rm can}(K_v)=1$.  We may then form the product measure $\mu_G^{\rm can}=\prod_v \mu_{G,v}^{\rm can}$ on $G(\A)$. All convolutions (local and global) on $G$ will be defined with respect to these measures. If $f \in C^\infty_c(G(\A))$, we define the operator $\pi(f)$ on $L^2(G(F) \backslash G(\A) )$ by
\[
[\pi(f)\phi](x) = \int_{G(\A)} \phi(xg) f(g) d\mu_G^{\rm can}(g).
\]
If $f \in C^\infty_c(G(\A))$, we define $f^*$ by $f^*(g) = \overline{f}(g^{-1})$, so that $\pi(f)$ and $\pi(f^*)$ are adjoints.

The choice of canonical measure for $G$ is imposed by the use of the Arthur--Selberg trace formula in Section \ref{sec:atf}; indeed one wants a uniform way of normalizing measures on the collection of connected reductive groups appearing as centralizers. We can afford to be more casual with measure normalizations for $H$, in light of our treatment of the geometric side of the relative trace formula in Section \ref{sec:artf}. For finite places $v$ we choose Haar measures $dh_v$ on $H(F_v)$ so that $H(F_v)\cap K_v$ is assigned measure 1. Because $H_v$ is compact for archimedean $v$, at these places we choose Haar measures $dh_v$ so that $H_v$ has volume 1.  We set $dh = \prod_v dh_v$.

\subsection{Hecke algebras}\label{Hecke-alg}

If $S$ is any finite set of finite places prime to $D$, let $\cH_S$ be the convolution algebra of functions on $G(F_S)$ that are compactly supported and bi-invariant under $K_S$.  Similarly, $\cH_f$ is the convolution algebra of functions on $G(\A_f^{D})$ that are compactly supported and bi-invariant under $K^{D}$. We identify $\cH_S$ and $\cH_f$ with subalgebras of $C^\infty_c(G(\A_f))$ in the natural way. If $G_\infty^0$ denotes the connected component of the identity in $G_\infty$ in the real topology, we define $\cH_\infty$ to be the subspace of $C^\infty_c(G_\infty^0)$ consisting of functions that are bi-invariant under $K_\infty$, and define $\cH = \cH_\infty \otimes \cH_f$.

To control for the degree of Hecke operators in our later estimates, it will be convenient to work with the truncated Hecke algebras defined in \cite{ST}. Recall from Section \ref{sec:AG} the Weyl-invariant norm $\| \cdot \|^*$ on $X_*(T)$.  Let $v$ be a finite place not dividing $D$.  Because $G_v$ is unramified, there is a maximal $F_v$-split torus $A_v$ in $G_v$ such that $K_v$ corresponds to a point in the apartment of $A_v$.  We may conjugate $A_v$ inside $T$ over $\overline{F}_v$, and obtain a norm $\| \cdot \|_v$ on $X_*(A_v)$.  We then define
\[
\cH_v^{\leqslant\kappa}={\rm span}_\C\,\{{\bf 1}_{K_v \mu(\varpi_v) K_v}:\mu\in X_*(A_v),\, \|\mu\|_v \leqslant\kappa\};
\]
note that here and elsewhere, ${\bf 1}_A$ denotes the characteristic function of the set $A$.  If $S$ is any finite set of finite places not dividing $D$, we set $\cH_S^{\leqslant \kappa} = \otimes_{v \in S} \cH_v^{\leqslant \kappa}$.  

We note that $\cH_v^{\leqslant\kappa}$ is independent of the choices made in its definition.  First, the norm $\| \cdot \|_v$ is independent of the conjugation of $A_v$ inside $T$.  To see this, we note that $\| \lambda \|^*$ can also be defined by taking $\mathbb{G}_m \overset{\lambda}{\longrightarrow} T \overset{\Ad}{\longrightarrow} GL(\g)$, decomposing the representation of $\mathbb{G}_m$ on $\g$ into characters in $X^*( \mathbb{G}_m ) \simeq \Z$, and taking half the sum of the positive terms.  This lets us extend $\| \cdot \|^*$ to all of $X_*(G)$ in a conjugation invariant way, and $\| \cdot \|_v$ is the restriction of this to $X_*(A_v)$.   Secondly, if $A_v$ and $A_v'$ are two tori whose apartments contain the point corresponding to $K_v$, then by \cite[Section 2.2.1]{Ti} they are conjugate by an element of $K_v$, and this conjugation respects the norms on $X_*(A_v)$ and $X_*(A_v')$.  This implies that $\cH_v^{\leqslant\kappa}$ is independent of $A_v$.

\subsection{A lemma on double cosets}

We shall need bounds on the volume of the support of Hecke operators on $G_v$ and $H_v$ at places where these groups are split.  Such bounds are proved by Macdonald \cite{Mc} for split simply connected semisimple groups, but we shall need them in the general split reductive case, and derive them now from a result of Gross \cite{Gr1}.

In this subsection, we let $F$ be a $p$-adic field with integer ring $\cO$, uniformizer $\varpi$, and residue field $k$, and let $q = \# k$.  We let $G$ be a split connected reductive group over $F$. Let $T \subset G$ be a maximal split torus, and $K \subset G$ a hyperspecial maximal compact subgroup corresponding to a point in the apartment of $T$. Let $\| \cdot \|^*$ be the function on $X_*(T)$ as in Section \ref{sec:AG}.

\begin{lemma}\label{coset-vol}
We have $\# K \lambda(\varpi) K / K \sim q^{2\| \lambda \|^*}$ for all $\lambda \in X_*(T)$, where the implied constants depend only on $\dim G$.
\end{lemma}

\begin{proof}
Let $P_\lambda$ be the parabolic subgroup associated with $\lambda$, so that
\[
\text{Lie}(P_\lambda) = \text{Lie}(T) + \bigoplus_{ \langle \alpha, \lambda \rangle \ge 0 } \text{Lie}(G)_\alpha.
\]
By Proposition 7.4 of \cite{Gr1}, we have $\# K \lambda(\varpi) K / K = q^{2 \| \lambda \|^* - \dim(G/P_\lambda)} \# (G/P_\lambda)(k)$. Now
\[
\# (G/P_\lambda)(k) = \# G(k) / \# P_\lambda(k)\quad\text{ and }\quad \# P_\lambda(k) = \# M_\lambda(k) q^{\dim U_\lambda},
\]
where $P_\lambda = M_\lambda U_\lambda$ is the standard Levi decomposition, and the first equality follows from \cite[Cor 16.5]{Bo}.  It follows that
\[
q^{- \dim(G/P_\lambda)} \# (G/P_\lambda)(k) = q^{\dim M_\lambda - \dim G} \# G(k) / \# M_\lambda(k).
\]
The formula of Steinberg \cite{St} for the number of points of a reductive group over a finite field implies that $\# G(k) / q^{\dim G} = f(q^{-1})$ and $\# M_\lambda(k) / q^{\dim M_\lambda} = f_\lambda(q^{-1})$, where $f$ and $f_\lambda$ are polynomials with constant term 1 that depend only on the root data of $G$ and $M_\lambda$ respectively.  This completes the proof.
\end{proof}

\subsection{Lie algebras}

Let $\g$ be the real Lie algebra of $G(F_\infty)$, and let $\g = \gk + \p$ be the Cartan decomposition associated with $K_\infty$. Let $\ga \subset \p$ be a maximal abelian subalgebra.  We let $\Delta_\R$ be the roots of $\ga$ in $\g$, and let $\Delta_\R^+$ be a choice of positive roots. We let $W_\R$ be the Weyl group of $\Delta_\R$. For $\alpha \in \Delta_\R$, we let $m(\alpha)$ denote the dimension of the corresponding root space. We denote the Killing form on $\g$ and $\g^*$ by $\langle \cdot, \cdot \rangle$, and let $\| \cdot \|$ be the associated norm on $\ga$ and $\ga^*$.  

Let $c(\xi)$ be Harish-Chandra's $c$-function; see \cite[Chapter II, $\mathsection$3.3]{He2}. Then it is known \cite[(3.44)]{DKV} that the upper bound
\begin{equation}\label{c-beta}
|c(\xi)|^{-2}\ll \beta(\xi)
\end{equation}
holds for all $\xi\in\ga^*$, where we have put
\[
\beta(\xi) = \prod_{\alpha \in \Delta^+_\R} ( 1 + | \langle \alpha, \xi \rangle| )^{m(\alpha)}.
\]
\subsection{Proof of Proposition \ref{admissibleexist}}\label{sec:Borel} We now recall Borel's theorem on the existence of rationally defined Cartan involutions on real semisimple Lie algebras. This is the central ingredient in the proof of Proposition \ref{admissibleexist}.

\begin{prop}[(Borel)]
\label{Flieexist}

Let $F$ be a totally real number field, and let $v_0$ be a real place of $F$.  Let $\g'$ be a real semisimple Lie algebra. There exists $\g / F$ with an involution $\theta$ defined over $F$ such that $\g_{v_0} \simeq \g'$, $\g_v$ is of compact type for all real $v \neq v_0$, and $\theta$ induces a Cartan involution of $\g_{v_0}$.

\end{prop}

\begin{proof}
The statement can be extracted from the proof of \cite[Proposition 3.8]{Borel63}. For completeness we provide details.

By \cite[Proposition 3.7]{Borel63}, there exists $\g_0/\Q$ and a $\Q$-involution $\theta$ of $\g_0$ such that $\g_0 \otimes \R \simeq \g'$, and $\theta$ induces a Cartan involution of $\g_0 \otimes \R$. Let $\g_0 = \gk_0 + \p_0$ be the rational Cartan decomposition associated with $\theta$. Choose $\alpha \in F$ such that $F=\Q(\alpha)$ and $\alpha_{v_0} > 0$, and $\alpha_v < 0$ for all other real $v$. Define $\g = \gk_0 \otimes F + \sqrt{\alpha}\, \p_0 \otimes F \subset \g_0 \otimes \overline{\Q}$. We extend $\theta$ to an involution of $\g$ defined over $F$.

We have $\g_{v_0} \simeq \g_0 \otimes \R \simeq \g'$, and it is clear that $\theta$ induces a Cartan involution of $\g_{v_0}$.  If $v \neq v_0$ is real, we have $\g_v = \gk_0 \otimes \R + i \p_0 \otimes \R \subset \g_0 \otimes \C$ so that $\g_v$ is of compact type.
\end{proof}

The corresponding statement for simply connected semisimple groups  -- Proposition \ref{admissibleexist} -- can be deduced from the Lie algebra version, as follows.

Let $G'$ be as in Proposition \ref{admissibleexist} and let $\g' / \R$ be its Lie algebra. Let $F$ be a totally real number field and $v_0$ a real place of $F$. Applying Proposition \ref{Flieexist} to $\g'$ with this choice of $F$ and $v_0$ we obtain a corresponding $\g / F$ and involution $\theta$.  Let $\text{Int}(\g)$ be the identity component of $\text{Aut}(\g)$, and let $G$ be the simply connected cover of $\text{Int}(\g)$.  Then $G$ has Lie algebra $\g$, and $\theta$ induces an involution of $G$ which we also denote by $\theta$ (as the derivative at the identity of the former is equal to the latter).  Because $G'$ and $G_{v_0}$ are connected, simply connected semisimple groups with isomorphic Lie algebras, they are isomorphic.  If $v \neq v_0$ is real, then $\g_v$ is of compact type and $G(F_v)$ is compact.

By Proposition \ref{Flieexist}, $\theta$ induces a Cartan involution of $\g_{v_0}$.  Because $G$ is simply connected, $G_{v_0}$ is connected in the real topology by the Cartan Connectedness Theorem \cite[Chapter 4, $\mathsection$2.2, Theorem 2.2]{OV}, and it follows that $\theta$ also induces a Cartan involution of $G_{v_0}$ (in the sense that the fixed point set $G_{v_0}^\theta$ is a maximal compact subgroup of $G_{v_0}$).  This completes the proof of Proposition \ref{admissibleexist}.

\subsection{Congruence subgroups and adelic quotients}\label{adelic-quotients}

Recall from Section \ref{sec:AG} the integer $D$, and from Section \ref{cmp-sg} the compact open subgroups $K_v$ for finite places $v$ and their product $K = \prod_{v \nmid \infty} K_v$ inside $G(\A_f)$. We now define congruence subgroups of $K_v$ and $K$ to be used throughout the paper.

Fix an integer $N\geqslant 1$ prime to $D$. For $v \nmid D \infty$ we write $K_v(N)$ for the level $N$ principal congruence subgroup, given by $K_v(N) = \rho^{-1}( \rho(G_v) \cap \SL_d(\cO_v)(N) )$. Here $\SL_d(\cO_v)(N)$ denotes the principal congruence subgroup of $\SL_d(\cO_v)$. Let $S$ be a finite set of finite places prime to $D$. Then we put
\[
K(N)=\prod_{v \nmid D \infty} K_v(N)\prod_{v \mid D}K_v \qquad\text{ and } \qquad K^S(N)=\prod_{\substack{v \nmid D \infty \\ v \notin S}} K_v(N) \prod_{v \mid D}K_v.
\]
For every finite $v$ we define $K_{H,v} = H_v \cap K_v$.  Letting
\[
K_H=\prod_{v \nmid \infty} K_{H,v} \qquad \text{ and } \qquad K_H^S=\prod_{v \notin S\cup\infty} K_{H,v},
\]
then the congruence manifold of interest to us is
\[
Y_N = G(F) \backslash G(\A) / K(N)K_H K_\infty.
\]
Let $\mathrm{Vol}_N$ be the volume assigned by $\prod_{v\nmid\infty} \mu_{G,v}^{\rm can}$ to the compact open subgroup $K(N)K_H$.

Finally, we write
\[
[G]=G(F) \backslash G(\A) \qquad\text{and}\qquad [H]=H(F) \backslash H(\A) 
\]
for the automorphic spaces associated with $G$ and $H$. We shall identify $L^2(Y_N)$ with the functions in $L^2([G])$ fixed under $K(N)K_HK_\infty$.

\subsection{Hecke--Maass forms}\label{Maass}

Let $\mathcal{D}$ be the algebra of differential operators on $G_\infty^0 / K_\infty$ that are invariant under the left action of $G_\infty^0$.  Note that if we define $K_{v_0}^+$ to be the group of fixed points of $\theta$ on $G_{v_0}$, and let $K_\infty^+$ be the group obtained by replacing $K_{v_0}$ with $K_{v_0}^+$ in the definition of $K_\infty$, then $K_\infty^+$ is a maximal compact subgroup of $G_\infty$, we have $G_\infty / K_\infty^+ \simeq G_\infty^0 / K_\infty$, and the elements of $\mathcal{D}$ are also invariant under the larger group $G_\infty$.  It follows that $\cD$ descends to an algebra of operators on $Y_N$ in a natural way.

We define a Hecke--Maass form to be a function on $Y_N$ that is an eigenfunction of the ring $\mathcal{D}$ on $Y_N$ and the Hecke algebra $\cH_f$ (and hence of $\cH$). If $\psi$ is a Hecke--Maass form and $\omega \in \cH$, then, using the notation in Section \ref{adelic-measures}, we define $\widehat{\omega}(\psi)$ by the equation $\pi(\omega) \psi = \widehat{\omega}(\psi) \psi$. We define the spectral parameter of $\psi$ to be the unique $\xi \in \ga^*_\C / W_\R$ such that $\psi$ has the same eigenvalues under the action of $\mathcal{D}$ as the associated spherical function $\varphi_\xi$. The Laplace eigenvalue of $\psi$ is given by $(\Delta + C_1(G) + \langle \xi, \xi \rangle) \psi = 0$ for some $C_1(G) \in \R$.

\section{A comparison of trace formulae}

In the next two sections, we state without proof a trace formula and relative trace formula, and deduce Theorems \ref{thm1} and \ref{thm2} from them. The trace formulae will be proved in Sections \ref{sec:artf} and \ref{sec:atf}.

\subsection{The trace formula}
\label{subsec:tf}

Let $N$ be a positive integer prime to $D$. Let $D'$ be a sufficiently large multiple of $D$, depending only on $G$, to be chosen in Section \ref{adelicOI}.  Let $T$ be a finite set of finite places prime to $D'$ and $N$. Let $\xi\in\mathfrak{a}^*$.  We will use the trace formula to count Hecke--Maass forms on $Y_N$ with spectral parameter near $\xi$, weighted by a Hecke operator at $T$.  To do this, we shall consider test functions of the form $\phi={\bf 1}^T_N\otimes k_T\otimes k_\xi$, for $k_T\in \mathcal{H}_T$ and $k_\xi\in \mathcal{H}_\infty$, where we have put
\be\label{def-eps}
{\bf 1}_N^T={\bf 1}_{K^T(N)K^T_H}.
\ee
We introduce the following condition on $k_\xi$:

\begin{enumerate}
\item[$({\mathsf P}_{\!\!\infty})$ :] $k_\xi$ is supported in $\{g\in G(F_\infty): d(g,H_{v_0}) <1\}$, and satisfies
\[
k_\xi(g)\ll \beta(\xi)(1 + \| \xi \| d(g,H_{v_0}) )^{-1/2}.
\]
\end{enumerate}
If $U$ is any finite set of finite places, we define $q_U = \prod_{v \in U} q_v$.  Our asymptotic trace formula is the following.

\begin{theorem}\label{trace-guy}
There are constants $A,B,\delta,\eta>0$ such that the following holds. Let $N$, $T$, and $\xi$ be as above, and let $\kappa\geqslant 0$. For any $k_T \in \cH_T^{\leqslant  \kappa}$ and any $k_\xi\in\mathcal{H}_\infty$ satisfying $(\mathsf{P}_\infty)$ we have

\bes
\mathrm{Vol}_N\sum_{i\geqslant 0} \widehat{k_T\otimes k_\xi}(\psi_i) =\mu^{\textnormal{can}}_G([G]) k_T(1)k_\xi(1) + O(N^{- \delta} q_T^{A\kappa + B}  \beta(\xi) (1 + \|\xi\|)^{-\eta}\| k_T \|_\infty),
\ees
where $\{ \psi_i \}$ runs over an orthonormal basis of Hecke--Maass forms for $L^2(Y_N)$. The implied constant depends only on $G$, $\theta$, and $K$.
\end{theorem}

We note that the exponents $A$, $B$, and $\delta$ are ineffective, and this is the only source of ineffectiveness of the exponent in Theorems \ref{thm1} and \ref{thm2}.  As discussed in Section \ref{sec:statement}, this comes from using theorems of Finis--Lapid and Cluckers--Gordon--Halupczok to bound orbital integrals.

\subsection{The relative trace formula}
\label{subsec:rtf}

For $\psi\in C^\infty (G(F) \backslash G(\A))$ we consider the $H$-automorphic period
\[
\mathcal{P}_H(\psi)=\int_{H(F) \backslash H(\A)} \psi(h) dh.
\]
Our relative trace formula will give the average size of these periods over an orthonormal basis $\{\psi_i\}$ of Hecke--Maass forms for $L^2(Y_N)$.

Let $N$ be a positive integer prime to $D$, $S$ a finite set of finite places prime to $D$ and $N$, and $\xi\in\mathfrak{a}^*$.  We now consider test functions of the form $\phi={\bf 1}^S_N\otimes k_S\otimes k_\xi$, for $k_S\in \mathcal{H}_S$ and $k_\xi\in \mathcal{H}_\infty$. As well as condition $({\mathsf P}_\infty)$, we shall use the following condition on $k_S$:
\begin{enumerate}
\item[$({\mathsf P}_S)$ :] there is $R>1$ such that $\|g\|_S\leqslant R$  for all $g\in {\rm supp}(k_S)$.
\end{enumerate}
The main term of our formula is expressed in terms of the averaging map
\[
\Pi_H: L^1(G(F_S))\longrightarrow L^1(G(F_S)/H(F_S))
\]
given by
\[
\Pi_H k_S(g)=\int_{H(F_S)}k_S(gh)dh.
\]
With this notation in hand, we may state our asymptotic relative trace formula.

\begin{prop}
\label{reltrace}
There is $A>0$ such that the following holds. Let $N$, $S$, and $\xi$ be as above. Let $k_S\in \mathcal{H}_S$ satisfy $({\mathsf P}_S)$ and $k_\xi\in \mathcal{H}_\infty$ satisfy  $({\mathsf P}_{\!\!\infty})$. Then
\begin{align*}
\mathrm{Vol}_N\sum_{i\geqslant 0}&\widehat{k_S\otimes k_\xi}(\psi_i) | \mathcal{P}_H(\psi_i)|^2\\
&= \mathrm{vol}([H])\Pi_H k_S(1) k_\xi(1)+O\left(\# ({\rm supp}\, k_S/K_S) \beta(\xi)(1 + \| \xi \| )^{-1/4} N^{-1/4} R^A \|k_S\|_{\infty}\right),
\end{align*}
where $\{\psi_i\}$ runs over an orthonormal basis of Hecke--Maass forms for $L^2(Y_N)$.

\end{prop}

\subsection{Proof of Theorems \ref{thm1} and \ref{thm2}}
\label{subsec:mainproof}

Recall from Section \ref{sec:lf} that $\cP$ denotes the set of finite places of $F$ that do not divide $D$ and at which $T$ splits. For $v \in \cP$ and $\mu \in X_*(T)$, we define $\tau(v,\mu) \in \cH_v$ to be the function supported on $K_v \mu( \varpi_v) K_v$ and taking the value $q_v^{-\| \mu \|^*}$ there.  Let $P$ be a positive integer and put
\be\label{def:S}
S=\{v \in \cP: P/2 \leqslant  q_v <  P, \; v \nmid N,\; v\nmid D' \}.
\ee
Choose any non-zero $\nu \in X_*(T)$, and let
\bes
\omega_S = \sum_{v\in S} \tau(v,\nu),
\ees
where, as in Section \ref{Hecke-alg}, we are identifying $\mathcal{H}_v$ with a subalgebra of $\mathcal{H}_S$ in the natural way.

In Section \ref{sec:avgperiod}, we will deduce the following weighted bounds from our two trace formulae.

\begin{prop}
\label{avgperiod}

For any $\nu$ and $P$, we have
\be
\label{avgupper}
\mathrm{Vol}_N\sum_{\|\Re \xi_i-\xi\|\leqslant 1} |\widehat{\omega_S}(\psi_i)|^2 \ll P \beta(\xi)(1+ P^A N^{-\delta}(1+\|\xi\|)^{-\delta}).
\ee
In addition, if $\nu \in X_*(T_H)$ satisfies the $H$-large inequality \eqref{K-large-cond}, then there is $Q > 0$ depending on $\nu$ such that
\be
\label{avglower}
\mathrm{Vol}_N\sum_{\|\Re \xi_i-\xi\|\leqslant Q} |\widehat{\omega_S}(\psi_i)|^2 | \mathcal{P}_H(\psi_i)|^2 \gg_\epsilon P^{2 - \epsilon} \beta(\xi)(1 + O( P^A N^{-\delta} (1 + \| \xi \|)^{-\delta}) ).
\ee
The exponent $A$, and the implied constants, depend on the choice of $\nu$.

\end{prop}

Theorems \ref{thm1} and \ref{thm2} follow easily from Proposition \ref{avgperiod}, as we now explain.  Choose $0 \neq \nu \in X_*(T_H)$ satisfying \eqref{K-large-cond}, and choose $P$ to be a small power of $(1+\|\xi\|) N$.  We then find $\delta>0$ such that for $(1 + \| \xi \|)N$ sufficiently large,
\be\label{upper}
\mathrm{Vol}_N\sum_{\|\Re \xi_i-\xi\|\leqslant Q} |\widehat{\omega_S}(\psi_i)|^2 \ll \beta(\xi)(1+\|\xi\|)^\delta N^\delta
\ee
and
\be\label{lower}
\mathrm{Vol}_N\sum_{\|\Re \xi_i-\xi\|\leqslant Q} |\widehat{\omega_S}(\psi_i)|^2 | \mathcal{P}_H(\psi_i)|^2 \gg_\epsilon \beta(\xi)(1+\|\xi\|)^{2\delta-\epsilon}N^{2\delta-\epsilon}.
\ee
Comparing \eqref{upper} and \eqref{lower} we find that there is $\delta>0$ and $Q>1$ such that for every $\xi\in\ga^*$ and every $N$ with $(1 + \| \xi \|)N$ large there is a Hecke--Maass form $\psi_i$ on $Y_N$ with spectral parameter $\|\xi_i-\xi\|\leqslant Q$ and satisfying $|\mathcal{P}_H(\psi_i)|\gg (1+\|\xi_i\|)^\delta N^\delta$. This implies the same lower bound on $\|\psi_i\|_\infty$, and in particular proves Theorem \ref{thm1}.

It remains to refine this power growth to obtain the stated lower bounds of Theorem \ref{thm2}. For this, we will make a special choice of orthonormal basis of Hecke--Maass forms for $L^2(Y_N)$. First recall that we have a Hilbert direct sum decomposition
\[
L^2(Y_N)=\bigoplus_\pi m(\pi)\pi^{K(N)K_HK_\infty},
\]
the sum ranging over irreducible representations of $G_\infty^0 \times G(\A_f)$ in $L^2(G(F)\backslash G(\A))$ satisfying $\pi^{K(N)K_HK_\infty} \neq 0$, each occuring with multiplicity $m(\pi)$. Each $\pi$ gives rise to a spectral datum $c_\pi$ (as defined in the paragraph preceding Theorem \ref{thm2}), and this assignment $\pi\mapsto c_\pi$ is finite to one. Given a spectral datum $c$ for $(G,N)$, the space $V(N,c)$ is the direct sum of $m(\pi)\pi^{K(N)K_HK_\infty}$ over all $\pi$ in the fiber over $c$.

Let $\{ c_i \}$ be an enumeration of the spectral data, and for each $c_i$ write $V_i = V(N,c_i)$ and $\xi_i = \xi(c_i)$.  The automorphic period $\mathcal{P}_H$ defines a linear functional on each finite dimensional vector space $V_i$, and its kernel is of codimension at most 1.  If the kernel is codimension 1 we let $\phi_i \in V_i$ be a unit normal to it, and otherwise choose $\phi_i \in V_i$ to be an arbitrary unit vector.  We may complete this set of vectors to an orthonormal basis of Hecke--Maass forms for $L^2(Y_N)$, to which we apply \eqref{upper} and \eqref{lower}.  Since $\mathcal{H}$ acts as a scalar on $V_i$ we may write the left-hand side of \eqref{upper} as
\[
\mathrm{Vol}_N\sum_{ \substack{c_i \\ \|\Re \xi_i-\xi\|\leqslant Q}} |\widehat{\omega_S}(\phi_i)|^2 \dim V_i,
\]
and the left-hand side of \eqref{lower} as
\[
\mathrm{Vol}_N\sum_{\substack{c_i\\ \|\Re \xi_i-\xi\|\leqslant Q}} |\widehat{\omega_S}(\phi_i)|^2 | \mathcal{P}_H(\phi_i)|^2.
\]
We obtain Theorem \ref{thm2} by comparing the right-hand sides of \eqref{upper} and \eqref{lower} as before.
\qed

\section{Bounds for average periods}
\label{sec:avgperiod}

We now deduce the average period bounds of Proposition \ref{avgperiod} from Theorem \ref{trace-guy} and Proposition \ref{reltrace}. For this, we must choose test functions $k_\xi\in\mathcal{H}_\infty$ and $k_S\in \mathcal{H}_S$ to insert into Theorem \ref{trace-guy} and Proposition \ref{reltrace} and explicate the error terms in those results.

\subsection{The archimedean test function}\label{k-inf}

Let $G_\infty^0$ denote the connected component of the identity in $G_\infty$ in the real topology.  If $\mu \in \ga^*_\C$, we define $\varphi_\mu$ to be the corresponding spherical function on $G_\infty^0$. If $k_\infty \in C^\infty_c(G_\infty^0)$, we define its Harish-Chandra transform by
\bes
\widehat{k}_\infty(\mu) = \int_{G_\infty^0} k_\infty(g) \varphi_{-\mu}(g)d\mu_{G,\infty}^{\rm can}(g).
\ees
Note that this is related to our earlier notation as follows: if $\psi$ is a Maass form with spectral parameter $\mu$ and $k_\infty$ is bi-$K_\infty$-invariant, then $\widehat{k}_\infty(\psi) = \widehat{k}_\infty(-\mu)$.

We shall choose $k_\xi$ so that its Harish-Chandra transform concentrates around $-\xi\in\ga^*$. For this we first take a function $h_0 \in C^\infty( \ga^* )$ of Paley--Wiener type that is real, nonnegative, and satisfies $h_0(0) = 1$.  We implicitly extend $h_0$ to a function on $\ga^*_\C$. Let
\[
h_\xi^0(\nu) = \sum_{w \in W_\R} h_0( w\nu - \xi),
\]
and let $k_\xi^0$ be the bi-$K_\infty$-invariant function on $G_\infty^0$ satisfying $\widehat{k}_\xi^0(-\mu) =h_\xi^0(\mu)$. We define $k_\xi = k_\xi^0 * k_\xi^0$ and $h_\xi = ( h_\xi^0 )^2$. We have $\widehat{k}_\xi(-\mu) =h_\xi(\mu)$, so that if $\psi$ has spectral parameter $\mu$ as before then $\widehat{k}_\xi(\psi) = h_\xi(\mu)$.  We define a unitary spectral parameter to be a spectral parameter of a spherical unitary representation of $G_\infty^0$.  We have $\| \Im \mu \| \leqslant \| \rho \|$ for any unitary spectral parameter $\mu$ \cite[Ch. IV, Thm 8.1]{HGGA}. If $\mu \in \ga^*_\C$ is unitary, Lemma \ref{realparam} implies that $h_\xi(\mu) = h_\xi^0(\mu) h_\xi^0(\overline{\mu}) = |h_\xi^0(\mu)|^2 \geqslant 0$.

\begin{lemma}
\label{realparam}

If $\mu \in \ga^*_\C$ is a unitary spectral parameter, then $W_\R \mu = W_\R \overline{\mu}$.

\end{lemma}

\begin{proof}

It suffices to show that $\varphi_\mu = \varphi_{\overline{\mu}}$.  Realizing $\varphi_\mu$ as a matrix coefficient of a unitary representation gives $\varphi_\mu(g) = \overline{\varphi}_\mu(g^{-1})$ for all $g$. The usual formula for $\varphi_\mu$ as a $K_\infty$-integral gives $\overline{\varphi}_\mu(g^{-1}) = \varphi_{-\overline{\mu}}(g^{-1})$. The identity $\varphi_{-\overline{\mu}}(g^{-1}) = \varphi_{\overline{\mu}}(g)$, obtained by e.g. setting $h = 1$ in Lemma 4.4 of \cite[Ch. IV]{HGGA} and comparing with Theorem 4.3 there, completes the proof.
\end{proof}

\begin{lemma}
\label{BPlemma}
The function $k_\xi$ satisfies property $({\mathsf P}_{\!\!\infty})$.\end{lemma}
\begin{proof}
Note that $k_\xi$ is of compact support independent of $\xi$ from the Paley--Wiener theorem of \cite{Ga}; we may thus take $h_0$ so that the support of $k_\xi$ lies in $\{g\in G(F_\infty): d(g,H_{v_0}) <1\}$ for all $\xi$. 

We have the inversion formula
\begin{equation}\label{HCinversion}
k_\xi(g) = \frac{1}{|W_\R|} \int_{\ga^*} \widehat{k}_\xi(\mu) \varphi_\mu(g) |c(\mu)|^{-2} d\mu.
\end{equation}
We can now quote either Theorem 2 of \cite{BP} or Proposition 7.2 of \cite{MT}, and apply our conditions on $G_\infty$, to find
\be
\label{Blomer}
\varphi_\mu(x) \ll_C (1 + \| \mu \| d(x, K_\infty) )^{-1/2},
\ee
for $\mu \in \ga^*$ and $x$ in a compact set $C \subset G_\infty$. In the inversion formula \eqref{HCinversion}, we apply \eqref{c-beta} to bound the Plancherel density function, \eqref{Blomer} to bound the spherical function, and recall the concentration of $\widehat{k}_\xi$ around $-\xi$, to obtain
\[
k_\xi(x) \ll \beta(\xi) (1 + \| \xi \| d(x, K_\infty) )^{-1/2},
\]
whence the claim.
\end{proof}

The following lemma will be helpful when truncating spectral sums.  We let $B(\mu) \subset \ga^*$ denote the unit ball around $\mu \in \ga^*$.

\begin{lemma}
\label{hpositive}

We may choose the function $h_0$ so that the functions $h_\mu$, $\mu \in \ga^*$, have the following property: if $\lambda$ is a unitary spectral parameter with $\Re \lambda \in B(\mu)$, then $h_\mu(\lambda) \ge c > 0$ for some $c$ depending only on $G$.

\end{lemma}

\begin{proof}

Let $0 < \delta < 1$.  Let $b \in C^\infty_0(\ga)$ be non-negative, supported in the $\delta$-ball around 0, and satisfy $\int b = 1$.  We further assume that $b = b_0 * b_0$ for some even real-valued $b_0$, which ensures that $\widehat{b}(\nu)$ is non-negative for $\nu \in \ga^*$. We will show that $h_0 = \widehat{b}$ suffices in our construction of $h_\mu$.

We let $C > 0$ be a constant depending only on $G$ that may vary from line to line.  We start by showing that $\Re \widehat{b}(\nu) \geqslant -C\delta$ for all $\nu \in \ga^*_\C$ with $\| \Im \nu \| \leqslant \| \rho \|$.  We have
\begin{align*}
\widehat{b}(\nu) & = \int_\ga b(H) e^{-i \nu(H)} dH \\
& = \int_\ga b(H) [ e^{-i \Re \nu(H)} + ( e^{-i \nu(H)} - e^{-i \Re \nu(H)} ) ] dH \\
& = \widehat{b}(\Re \nu) + \int_\ga b(H)  e^{-i \Re \nu(H)} ( e^{\Im \nu(H)} - 1 ) dH.
\end{align*}
As $\| \Im \nu \| \leqslant \| \rho \|$, we have $| e^{\Im \nu(H)} - 1 | \leqslant C \delta$ for all $H \in \text{supp}(b)$, so that
\[
\Re \widehat{b}(\nu) \geqslant \widehat{b}(\Re \nu) - C\delta \geqslant -C\delta
\]
as required.

We now take $h_0 = \widehat{b}$, and construct $h_\mu^0$ and $h_\mu$ as before.  If we choose $\delta$ small, we will have $\Re h_0(\nu) \geqslant 1/2$ for all $\nu$ with $\| \Re \nu \| \leqslant 1$, $\| \Im \nu \| \leqslant \| \rho \|$.  Therefore if $\lambda \in \ga^*_\C$ is unitary with $\Re \lambda \in B(\mu)$, the bound $\| \Im \lambda \| \le \| \rho \|$ implies that
\begin{align*}
\Re h_\mu^0(\lambda) & = \sum_{w \in W_\R} \Re h_0(w\lambda - \mu) \\
& \geqslant \Re h_0(\lambda- \mu) - C \delta \geqslant 1/4.
\end{align*}
This gives $h_\mu(\lambda) = | h_\mu^0(\lambda) |^2 \geqslant 1/16$ as required.
\end{proof}

\subsection{The $S$-adic test function}\label{kS}

We choose the test function at places in $S$ to be $k_S=\omega_S\omega_S^*\in \mathcal{H}_S$, where $\omega_S$ is as in Section \ref{subsec:mainproof}.

\begin{lemma}\label{Hecke-fix}
The function $k_S$ satisfies the following.

\begin{enumerate}
\smallskip
\item[{\it (a)}] There is $B>0$ such that $\|g\|_S\ll P^B$ for all $g\in {\rm supp}(k_S)$.
\smallskip
\item[{\it (b)}] We have $\|k_S\|_\infty \ll P$.
\smallskip
\item[{\it (c)}] There is $C>0$ such that $\# ({\rm supp}\, k_S/K_S)\ll P^C$.
\end{enumerate}
The exponents $B$ and $C$ depend on the underlying choice of $\nu$ in the definition of $k_S$. All implied constants depend on $G$ and $\nu$.
\end{lemma}

\begin{proof}
(a)  Because supports add under convolution, it suffices to show this for $g \in \text{supp}(\tau(v,\nu))$ (or $\text{supp}(\tau(v,-\nu))$, which is similar and we omit) for some $v \in S$.  Our assumption that $\rho(K_w) \subset \SL_d(\cO_w)$ for all finite $w$ means that we only need to estimate $\| \rho(g) \|_v$.  Consider $\rho$ as a representation of $G$, and let $\Omega \subset X^*(T)$ be the multiset of weights of this representation.  It follows that $\rho(\nu(\varpi_v))$ is semisimple with eigenvalues $\{ \varpi_v^{\langle \omega, \nu \rangle } : \omega \in \Omega \}$.  By \cite[Lemma 2.17]{ST}, there is $x \in \GL_d(\cO_v)$ such that $x \rho(T_v) x^{-1}$ is diagonal, and so if we define $A = \max \{ - \langle \omega, \nu \rangle : \omega \in \Omega \}$ then we have $\| \rho(\nu(\varpi_v)) \|_v = \| x \rho(\nu(\varpi_v)) x^{-1} \|_v = q_v^A < P^A$.  Because $\rho(G(\cO_v)) \subset \SL_d(\cO_v)$, the same holds for $g \in \text{supp}(\tau(v,\nu))$.

\medskip

(b) We note that
\be\label{kS-expand}
k_S= \sum_{v\in S} \tau(v,\nu)\tau(v,\nu)^* + \sum_{\substack{v, w\in S\\ v\neq w}} \tau(v,\nu)\tau(w,\nu)^*.
\ee
The bound is clear for the second sum, because the terms satisfy $\|\tau(v,\nu)\tau(w,\nu)^* \|_\infty\leqslant 1$ and the supports of the terms are disjoint.  For the first sum, we have
\[
\| \tau(v,\nu)\tau(v,\nu)^* \|_\infty \leqslant  \| \tau(v,\nu) \|_2^2,
\]
and $\| \tau(v,\nu) \|_2 \ll 1$ follows from Lemma \ref{coset-vol}.

\medskip

(c)  We may write the first sum in \eqref{kS-expand} as a linear combination of $\tau(v,\mu)$ with $\mu$ lying in a finite set depending on $\nu$.  The bound now follows from Lemma \ref{coset-vol}.  The second sum may be treated similarly.
\end{proof}

\subsection{Proof of Proposition \ref{avgperiod} }\label{PfThm2}

We are now in a position to prove Proposition \ref{avgperiod}.  We first note that $k_\xi(1)\asymp \beta(\xi)$, where the upper and lower bounds follow from property ($\mathsf{P}_\infty$) and \cite[(3.44a)]{DKV}, respectively. Moreover, $\beta(\xi)$ is bounded above (and below) by a power of $(1+\|\xi\|)$.

We now apply Proposition \ref{reltrace} with test functions $k_S k_\xi$. Moreover, for every pair $v,w\in S$ we put $T=\{v,w\}$ and apply Theorem \ref{trace-guy} with test function $k_T k_\xi$ where $k_T=\tau(v,\nu)\tau(w,\nu)^*$; we then sum over such pairs $v,w$. As a result, we deduce the existence of constants $A>2$ (depending on $\nu$) and $\delta>0$ such that
\be\label{trace-error}
\mathrm{Vol}_N\sum_{i\geqslant 0} |\widehat{\omega_S}(\psi_i)|^2h_\xi(\xi_i) \asymp\omega_S\omega_S^*(1) \beta(\xi)+ O( P^A  N^{- \delta}\beta(\xi)(1+\|\xi\|)^{-\delta})
\ee
and
\be\label{relative-error}
\mathrm{Vol}_N\sum_{i\geqslant 0} |\widehat{\omega_S}(\psi_i)|^2 |\mathcal{P}_H(\psi_i)|^2h_\xi(\xi_i)\asymp \Pi_H \omega_S\omega_S^*(1)\beta(\xi)+O(P^A N^{-\delta}\beta(\xi)(1+\|\xi\|)^{-\delta}).
\ee
The error term in \eqref{trace-error} was obtained by observing that $\tau(v,\nu)\tau(w,\nu)^*\in\mathcal{H}_T^{\leqslant \kappa}$, for some $\kappa$ depending only on $\nu$, and $q_T=q_v q_w< P^2$, so that $q_T^{A\kappa +B}$ is bounded by a power of $P$; we must also insert the $L^\infty$ norm estimate for $k_T$ coming from the proof of Lemma \ref{Hecke-fix}. The error term in \eqref{relative-error} was obtained by taking $B$ as in Lemma \ref{Hecke-fix} and setting $R=P^B$ in condition $({\mathsf P}_S)$, and inserting the $L^\infty$ norm estimate of Lemma \ref{Hecke-fix}.  It remains then to explicate the size of $\omega_S\omega_S^*(1)$ and $\Pi_H \omega_S\omega_S^*(1)$ (the latter under our extra condition on $\nu$), and truncate the spectral sums. We continue to use the convention that the values of the exponents $A,\delta>0$ can vary from line to line.

We begin with \eqref{trace-error}. Lemma \ref{Hecke-fix} (b) gives $\omega_S\omega_S^*(1)\ll P$. By positivity we may truncate the spectral sum to obtain the upper bound
\[
\mathrm{Vol}_N\sum_{\|\Re \xi_i-\xi\|\leqslant 1} |\widehat{\omega_S}(\psi_i)|^2 h_\xi (\xi_i)\ll P \beta(\xi)(1+ P^A N^{-\delta}(1+\|\xi\|)^{-\delta}).
\]
Combined with Lemma \ref{hpositive}, this implies \eqref{avgupper}.

We next examine \eqref{relative-error}. We begin by observing that if $\mu \in \ga^*$, then we may again truncate the spectral sum and apply Lemma \ref{hpositive} to obtain
\be
\label{spectral-error}
{\rm Vol}_N\sum_{\|\Re \xi_i-\mu\|\leqslant 1} |\widehat{\omega_S}(\psi_i)|^2 | \mathcal{P}_H(\psi_i)|^2 \ll \Pi_H \omega_S\omega_S^*(1) \beta(\mu) + P^A \beta(\mu) (1 + \| \mu \|)^{-\delta} N^{-\delta}.
\ee
We next use this to truncate the spectral sum about $\xi$. 

\begin{lemma}
\label{relative-truncate}

For any $Q>1$ the left-hand side of \eqref{relative-error} can be written as
\begin{align*}
\mathrm{Vol}_N\sum_{\| \Re \xi_i-\xi\|\leqslant Q} |\widehat{\omega_S}&(\psi_i)|^2 | \mathcal{P}_H(\psi_i)|^2h_\xi(\xi_i)\\
& + O_M( \Pi_H \omega_S\omega_S^*(1) \beta(\xi) Q^{-M})+O(P^A N^{-\delta}  \beta(\xi)(1 + \| \xi \|)^{-\delta}).
\end{align*}

\end{lemma}

\begin{proof}
We must show that the sum
\[
\mathfrak{S}={\rm Vol}_N\sum_{\|\Re \xi_i-\xi\|> Q} |\widehat{\omega_S}(\psi_i)|^2 | \mathcal{P}_H(\psi_i)|^2h_\xi(\xi_i)
\]
satisfies
\[
\mathfrak{S}=O_M( \Pi_H \omega_S\omega_S^*(1) \beta(\xi) Q^{-M})+O(P^A N^{-\delta}  \beta(\xi)(1 + \| \xi \|)^{-\delta}).
\]
Break the region in the positive chamber $\ga^*_+$ defined by $\|\mu-\xi\|> Q$ into an overlapping union of unit balls $B(\mu_n)$ centered at points $\mu_n\in\ga^*$.  Because unitary spectral parameters $\mu$ satisfy $\| \Im \mu \| \leqslant \| \rho \|$, they also satisfy the rapid decay estimate $h_\xi(\mu)\ll_{M} \| \Re \mu-\xi\|^{-M}$.  Applying this on each ball gives
\[
\sum_{\Re \xi_i\in B(\mu_n)} |\widehat{\omega_S}(\psi_i)|^2 | \mathcal{P}_H(\psi_i)|^2h_\xi (\xi_i)\ll_{M} \|\mu_n-\xi\|^{-M}\sum_{\Re \xi_i\in B(\mu_n)} |\widehat{\omega_S}(\psi_i)|^2 | \mathcal{P}_H(\psi_i)|^2.
\]
Applying \eqref{spectral-error} and summing over $n$ we obtain
\[
\mathfrak{S}\ll_{M} \Pi_H \omega_S\omega_S^*(1) \sum_n \|\mu_n-\xi\|^{-M} \beta(\mu_n) 
+ P^A N^{-\delta} \sum_n \|\mu_n-\xi\|^{-M} (1 + \| \mu_n \|)^{-\delta} \beta(\mu_n).
\]
From $\beta(\mu_n) \ll \| \mu_n - \xi \|^k \beta(\xi)$, where $k$ is the number of roots of $G$ counted with multiplicity, we may simplify this to
\[
\mathfrak{S}\ll_{M} \Pi_H \omega_S\omega_S^*(1) \beta(\xi) \sum_n \|\mu_n-\xi\|^{-M} 
+ \beta(\xi) P^A N^{-\delta} \sum_n \|\mu_n-\xi\|^{-M} (1 + \| \mu_n \|)^{-\delta}.
\]
The first sum is $\ll_M Q^{-M}$. We bound the second sum by breaking it into $\| \mu_n \| \leqslant \| \xi \|/2$ and the complement. The first sum is $\ll (1 + \| \xi \|)^{-M}$, and the second is $\ll Q^{-M} (1 + \| \xi \|)^{-\delta}$. Both of these are dominated by $(1 + \| \xi \|)^{-\delta}$, which finally shows that
\[
\mathfrak{S}\ll_{M} \Pi_H \omega_S\omega_S^*(1) \beta(\xi) Q^{-M} +  P^A N^{-\delta} \beta(\xi)(1 + \| \xi \|)^{-\delta},
\]
as desired.
\end{proof}

We combine Lemma \ref{relative-truncate} (taking any $M>0$ and large enough $Q$) and \eqref{relative-error} to obtain
\[
\mathrm{Vol}_N\sum_{\|\Re \xi_i-\xi\|\leqslant Q} |\widehat{\omega_S}(\psi_i)|^2 | \mathcal{P}_H(\psi_i)|^2h_\xi(\xi_i) \asymp \Pi_H \omega_S\omega_S^*(1)\beta(\xi) + O( P^A N^{-\delta}\beta(\xi) (1 + \| \xi \|)^{-\delta}).
\]
We now make use of the critical assumptions that $G$ is $H$-large, and $\nu\in X_*(T_H)$ satisfies \eqref{K-large-cond}, to bound from below the right hand side.
\begin{lemma}\label{nu-choice}
If $\nu\in X_*(T_H)$ satisfies \eqref{K-large-cond} then $\Pi_H \omega_S\omega_S^*(1)\gg_\epsilon P^{2-\epsilon}$.
\end{lemma}

\begin{proof}
Note that for any $v\in\mathcal{P}$ and $\nu\in X_*(T_H)$ we have
\[
q_v^{\|\nu\|^*}\int_{H_v} \tau(v,\nu)(x) dx= \text{vol}( H_v \cap K_v \nu(\varpi_v) K_v)  \geqslant \text{vol}( K_{H,v} \nu(\varpi_v) K_{H,v})\gg q_v^{2 \| \nu \|_H^*},
\]
where we have used our assumption that $T$, and hence $T_H$, is split at $v$, and applied the lower bound from Lemma \ref{coset-vol}. If $\nu\in X_*(T_H)$ satisfies \eqref{K-large-cond}, then so does $-\nu$, and applying the above bound with $\pm\nu$, we get $\int_{H(F_S)} \tau(v,\nu)\tau(w,\nu)^* \gg 1$ if $v \neq w$. Summing over $v,w\in S$ yields the claim. 
\end{proof}

We deduce from the above lemma that
\[
\mathrm{Vol}_N\sum_{\|\Re \xi_i-\xi\|\leqslant Q} |\widehat{\omega_S}(\psi_i)|^2 | \mathcal{P}_H(\psi_i)|^2h_\xi(\xi_i) \gg_\epsilon P^{2 - \epsilon} \beta(\xi)(1 + O( P^A N^{-\delta} (1 + \| \xi \|)^{-\delta}) ).
\]
The bound $h_\xi(\xi_i) \ll 1$ completes the proof of \eqref{avglower}.

\section{The amplified relative trace formula}\label{sec:artf}

We now prove Proposition \ref{reltrace}. Recall the notation introduced in Section \ref{subsec:rtf}.  For $\phi={\bf 1}_N^S\otimes k_S\otimes k_\xi$ we let
\[
K(x,y)=\sum_{\gamma \in G(F) } \phi(x^{-1} \gamma y)\qquad\text{and}\qquad K_H(x,y)=\sum_{\gamma \in H(F) } \phi(x^{-1} \gamma y).
\]
Integrating the spectral expansion
\[
K(x,y)= \mathrm{Vol}_N\sum_{i\geqslant 0} \widehat{k_S\otimes k_\xi}(\psi_i) \psi_i(x) \overline{\psi_i(y)}
\]
over $[H] \times [H]$, we obtain
\bes
\int_{[H] \times [H]} K(x,y) dx dy = \mathrm{Vol}_N\sum_{i\geqslant 0} \widehat{k_S\otimes k_\xi}(\psi_i) |\mathcal{P}_H(\psi_i)|^2.
\ees
On the other hand, by unfolding we have
\[
\int_{[H] \times [H]} K_H(x,y)dx dy = \mathrm{vol}([H])\int_{H(\A)} \phi(x) dx= \mathrm{vol}([H])\Pi_Hk_S(1) k_\xi(1).
\]
We have used the fact that $K^S(N)K^S_H\cap H(\A_f^S)=K^S_H$ and the volume of this is ${\rm vol}\, K_H^S=1$.

For the remaining terms, first observe that if $x, y \in H(\A)$ we have 
\be
\label{offdiagcount}
\# \{ G(F)\cap {\rm supp}( \phi(x^{-1} \cdot y) )\} = \# \{ x^{-1} G(F) y \cap {\rm supp}(\phi)\}\ll \# ({\rm supp}\, k_S/K_S),
\ee
uniformly in $x$, $y$, and $N$. Note that the left hand side of this bound only depends on the image of $x$ and $y$ in $[H]$, so we may assume that they lie in a fixed compact set.  To prove \eqref{offdiagcount}, let $g, g' \in x^{-1} G(F) y \cap {\rm supp}(\phi)$ with $g = x^{-1} \gamma y$, $g' = x^{-1} \gamma' y$, and suppose that $g, g'$ lie in the same coset in ${\rm supp}\, k_S/K_S$. Then $g_\infty^{-1}g_\infty'$ lies in a fixed compact set and $g_f^{-1}g_f'\in K$, so that $g^{-1} g'$ also lies in a fixed compact set. We have $g^{-1} g' = y^{-1} \gamma^{-1} \gamma' y$, so that $\gamma^{-1} \gamma'$ lies in a fixed compact set, and so if $g$ is fixed there are only finitely many possibilities for $g'$.  Therefore the map $x^{-1}G(F)y \cap {\rm supp}(\phi) \rightarrow  {\rm supp}\, k_S/K_S$ has $O(1)$ fibers, which proves \eqref{offdiagcount}. Using this, we simply estimate
\bes
\int_{[H] \times [H]} \sum_{\gamma \in G(F) - H(F) } \phi(x^{-1} \gamma y) dx dy 
\ees
with the pointwise bounds of Corollary \ref{fsmall}.

\subsection{Bounding the off-diagonal contributions}
In this section we establish Corollary \ref{fsmall}, which was used in the proof of Proposition \ref{reltrace} above. It is based on the following Diophantine lemma which shows, roughly speaking, that any $\gamma \in G(F) - H(F)$ cannot be too close to $H_v$ for various $v$.  We shall use the notation introduced in Section \ref{sec:metrics}.

\begin{lemma}\label{dioph}
If $C > 0$ is given, there are $A, C_1, C_2 > 0$ such that the following properties hold for any $\gamma \in G(F) - H(F)$:
\begin{enumerate}
\item\label{dioph1}  $d(\gamma, H_{v_0}) \geqslant C_1 \| \gamma \|_f^{-A}$;

\item \label{dioph2} if $(N,D) = 1$ is such that $N > C_2 \| \gamma \|_f^A$, and $d(\gamma, H_{v_0}) < C$, then there is a place $v | N$ such that $\gamma_v \notin K(N)_v K_{H,v}$.
\end{enumerate}
\end{lemma}

\begin{proof}
We consider $H$ and $G$ as subvarieties of $F^{d^2}$ via the embedding $\rho$. Let $p_1, \ldots, p_k \in \cO[1/D, x_1, \ldots, x_{d^2}]$ be a set of defining polynomials for $H$ that are integral over $\cO[1/D]$. Now if $\gamma \in G(F) - H(F)$ then $p_i(\gamma) \neq 0$ for some $i$.

There are $A, C_3 > 0$ such that for all $\gamma \in G(F)$ one has $\prod_{v\nmid\infty}|p_i(\gamma)|_v \leqslant C_3 \| \gamma \|_f^{A}$. For $\gamma \in G(F) - H(F)$ the product formula applied to $p_i(\gamma)\in F^\times$ implies that the archimedean norms satisfy $\prod_{v \mid \infty} | p_i(\gamma) |_v \geqslant \| \gamma \|_f^{-A} /C_3$. Because $G_v$ is compact for all $v \mid \infty$ other than $v_0$, $| p_i(\gamma) |_v$ is bounded for all such $v$.  But then we have $| p_i(\gamma) |_{v_0} \geqslant C_1 \| \gamma \|_f^{-A}$ for some $C_1 > 0$, and so $d(\gamma, H_{v_0})$ satisfies the same bound. This establishes \eqref{dioph1}.

As above there are $A_1, C_3 > 0$ such that $\prod_{v \nmid \infty N} | p_i(\gamma) |_v \leqslant C_3 \| \gamma \|_f^{A_1}$ for all $\gamma\in G(F)$. Moreover, we have $| p_i(\gamma)|_v \ll 1$ for $v | \infty$ by our assumption $d(\gamma, H_{v_0}) < C$. Now $p_i$ descends to a map $K_v / K(N)_v \rightarrow \cO_v / N \cO_v$ which is trivial on $K_{H,v}$. Thus if $\gamma_v \in K(N)_v K_{H,v}$ then $| p_i(\gamma) |_v \leqslant | N |_v$. If we then suppose that $\gamma_v \in K(N)_v K_{H,v}$ for all $v | N$ it would follow that $\prod_{v | N} | p_i(\gamma) |_v \leqslant N^{-|F:\Q|}$.  But if $N > C_2 \| \gamma \|_f^A$ for suitable $A, C_2 > 0$, we obtain a contradiction by again applying the product formula to $p_i(\gamma)\in F^\times$. This establishes \eqref{dioph2}.
\end{proof}

\begin{cor}
\label{fsmall}
There is $A>0$ such that the following holds. Let $N$, $S$, and $\xi$ be as in Section \ref{subsec:rtf}. Let $k_S\in \mathcal{H}_S$ satisfy $({\mathsf P}_S)$ and $k_\xi\in \mathcal{H}_\infty$ satisfy $({\mathsf P}_{\!\!\infty})$. Put $\phi={\bf 1}_N^S\otimes k_S\otimes k_\xi$. Then for all $\gamma \in G(F) - H(F)$ and all $x, y \in H(\A)$, we have
\[
\phi(x^{-1} \gamma y) \ll  \beta(\xi)(1 + \| \xi \|)^{-1/4} N^{-1/4} R^A \| k_S \|_\infty.
\]
\end{cor}

\begin{proof}
Let $\Omega_H \subset H(\A)$ be a compact set containing a fundamental domain for $[H]$.  We assume that $\Omega_H = \Omega_{H,D\infty} \times \prod_{v \nmid D\infty} K_{H,v}$ after possibly enlarging $D$.  Because $G(F) - H(F)$ is bi-invariant under $H(F)$, we may assume that $x, y \in \Omega_H$.

We may also clearly assume that $\phi(x^{-1} \gamma y) \neq 0$. It then follows from Property $({\mathsf P}_S)$ that $\| x^{-1} \gamma y \|_S \leqslant R$; in fact we have $\| x^{-1} \gamma y \|_f \leqslant R$, using the condition $x^{-1} \gamma y\in K^S(N)K_H^S\subset K^S$. When combined with $x, y \in \Omega_H$ this gives $\| \gamma \|_f \ll R$.  We may now apply part \eqref{dioph1} of Lemma \ref{dioph}, to find that $d(\gamma, H_{v_0}) \gg R^{-A}$. All together, since $x, y \in H_{v_0}$, we deduce that $d(x^{-1} \gamma y, H_{v_0}) \gg R^{-A}$. Similarly, from $k_\xi(x^{-1} \gamma y)\neq 0$ and Property $({\mathsf P}_{\!\!\infty})$ it follows that $d(x^{-1} \gamma y, H_{v_0}) < 1$, and hence $d(\gamma, H_{v_0}) \ll 1$.

Suppose that $1 + \|\xi \| \geqslant N$.  We have $(k_S\otimes k_\xi)(x^{-1} \gamma y) \leqslant \| k_S \|_\infty k_\xi(x^{-1} \gamma y)$. We then combine $({\mathsf P}_{\!\!\infty})$ with  $d(x^{-1} \gamma y, H_{v_0}) \gg R^{-A}$ to get
\begin{align*}
k_\xi(x^{-1} \gamma y) &  \ll \beta(\xi) (1 + \| \xi \| d(x^{-1} \gamma y, H_{v_0}) )^{-1/2}\\
& \ll \beta(\xi) (1 + \| \xi \| R^{-A} )^{-1/2} \\
& \ll   \beta(\xi) (1 + \| \xi \|)^{-1/2} R^{A/2} \\
& \ll   \beta(\xi) (1 + \| \xi \|)^{-1/4} N^{-1/4} R^{A/2},
\end{align*}
which completes the proof in this case.

Now suppose that $1 + \|\xi \| < N$.  Because $\| \gamma \|_f \ll R$, part \eqref{dioph2} of Lemma \ref{dioph} implies that there are $A, C > 0$ such that if $N > CR^A$, then there is a place $v | N$ for which $\gamma_v \notin K(N)_v K_{H,v}$.  Because $x, y \in \Omega_H$, we have $x, y \in K_{H,v}$, and so $x^{-1} \gamma y \notin K(N)_v K_{H,v}$.  It follows that if $N > CR^A$, then $\phi(x^{-1} \gamma y) = 0$.  We may rephrase this as saying that
\[
\phi(x^{-1} \gamma y) \leqslant \| k_S k_\xi\|_\infty N^{-1} C R^A\ll \| k_S\|_\infty \beta(\xi)N^{-1} R^A,
\]
and the bound $N^{-1} \leqslant N^{-1/2} (1 + \| \xi \|)^{-1/2}$ completes the proof.
\end{proof}

\section{The amplified trace formula}\label{sec:atf}

We now prove Theorem \ref{trace-guy}, our trace formula asymptotic with uniform error term. Throughout this section, we can and will relax our condition that $G_{v_0}$ is $\R$-almost simple to the condition that $G$ is $F$-almost simple.

Our proof relies crucially on recent work of Shin--Templier and Cluckers--Gordon--Halupczok \cite{ST} on bounding centralizer volumes and $p$-adic orbital integrals as well as work of Finis--Lapid \cite{FL} bounding intersection volumes of conjugacy classes with congruence subgroups. We must supply our own bounds on archimedean orbital integrals; these are proven in Section \ref{ROI}.  

\subsection{Canonical and Tamagawa measures}
\label{sec:Grossmeasures}

If $G$ is a general connected reductive group over $F$, Gross \cite[(1.5)]{Gr} attaches to $G$ an Artin--Tate motive
\bes
M_G = \bigoplus_{d \geqslant 1} M_{G,d}(1-d)
\ees
with coefficients in $\Q$.  Here $M_{G,d}$ is an Artin motive and $(1-d)$ denotes the Tate twist.  We let $\epsilon(M_G)$ be the $\epsilon$-factor of this motive, which is given by 

\bes
\epsilon(M_G) = |d_F|^{\dim G / 2} \prod_{d \geqslant 1} \mathcal{N}_{F/\Q}(\mathfrak{f}(M_{G,d}) )^{d-1/2},
\ees
where $\mathfrak{f}(M_{G,d})$ denotes the conductor of the Artin motive $M_{G,d}$ (see \cite[(9.8)]{Gr}).  We let $L( M_{G_v}^\vee(1))$ denote the $L$-function of the local motive $M_{G_v}^\vee(1)$, and $L(M_G^\vee(1))$ and $\Lambda(M_G^\vee(1))$ denote the finite and completed $L$-functions of $M_G^\vee(1)$. Then $L( M_{G_v}^\vee(1))$ is a positive real number, and $L(M_G^\vee(1))$ and $\Lambda(M_G^\vee(1))$ are finite if $Z(G)$ does not contain an $F$-split torus (see \cite[Proposition 9.4]{Gr}).  From now on we shall assume that $G$ satifies this condition.

In \cite[$\mathsection$11]{Gr} Gross defines a canonical measure $|\omega_{G_v}|$ on $G_v$ at any place of $F$.  We define $\mu_{G,v}^\text{can} = L( M_{G_v}^\vee(1)) \cdot | \omega_{G_v}|$ as in \cite{Gr}.  When $v$ is finite and $G$ is unramified at $v$, $\mu_{G,v}^\text{can}$ assigns volume 1 to a hyperspecial subgroup of $G(F_v)$, and so we can define the measure $\mu^\text{can}_G = \prod_v \mu^\text{can}_{G,v}$ on $G(\A)$.

Now let $\omega$ be a nonzero differential form of top degree on $G$ defined over $F$. For each $v$, one may associate with $\omega$ a Haar measure $|\omega|_v$ on $G(F_v)$. For almost all $v$, $L( M_{G_v}^\vee(1)) \cdot | \omega|_v$ assigns volume 1 to a hyperspecial subgroup of $G(F_v)$. Let $\mu^\text{Tam}_G$ be the Tamagawa measure on $G(\A)$, which is defined by
\[
\mu^\text{Tam}_G = \Lambda(M_G^\vee(1))^{-1} |d_F|^{-\dim G / 2}\bigotimes_v L(M_{G_v}^\vee(1)) |\omega|_v 
\]
(see \cite[(10.2)]{Gr}) and satisfies \cite[p. 629]{Kott2}
\bes
\mu^\text{Tam}_G(G(F)\backslash G(\A)) = | \pi_0( Z(\widehat{G})^{\text{Gal}(\overline{F}/F)} ) |  | \ker^1(F, Z(\widehat{G}) ) |^{-1}.
\ees
The comparison between $\mu_G^\text{can}$ and $\mu_G^\text{Tam}$ is given by \cite[Theorem 11.5]{Gr}, 
\be\label{cantam}
\frac{\mu^\textup{can}_G }{ \mu^\textup{Tam}_G } = \epsilon(M_G)\Lambda(M_G^\vee(1)).
\ee

\subsection{The trace formula}

We now assume that $G$ is anisotropic.  The trace formula is a distributional identity
\bes
I_\text{spec}(\phi, \mu^\text{can}_G) = I_\text{geom}(\phi, \mu^\text{can}_G),
\ees
for $\phi\in C_c^\infty(G(\A))$. More precisely,
\bes
I_\text{spec}(\phi, \mu^\text{can}_G) = \sum_\pi m(\pi) \tr( \pi(\phi) ),
\ees
where $\pi$ runs over all irreducible subrepresentations of $L^2(G(F)\backslash G(\A))$ occuring with multiplicity $m(\pi)$, and
\bes
I_\text{geom}(\phi, \mu^\text{can}_G) = \sum_{ \{ \gamma \} } \frac{\mu^\text{can}_{I_\gamma}(I_\gamma(F)\backslash I_\gamma(\A))}{ | G_\gamma(F) : I_\gamma(F)|} O_\gamma(\phi),
\ees
where $\{\gamma\}$ runs over all $G (F)$-conjugacy classes, $G_\gamma$ is the centraliser of $\gamma$ in $G$, $I_\gamma$ is the connected component of $G_\gamma$, and 
\[
O_\gamma(\phi)=\int_{I_\gamma(\A)\backslash G(\A)}\phi(x^{-1}\gamma x)d\mu_\gamma (x).
\]
(See e.g. \cite[Section 9.1]{Kott} for this formulation of the geometric side.)  The measure $\mu_\gamma$ above denotes the quotient measure $d\mu^\text{can}_G / d\mu^\text{can}_{I_\gamma}$.

We shall bound the terms in $I_\text{geom}(\phi, \mu^\text{can}_G)$ using the Weyl discriminant.  For any $v$ and $\gamma \in G_v$, this is defined by
\[
D_v(\gamma) = | \det(1 - \Ad(\gamma)_{|\g_v / \g_{v,\gamma}}) |_v,
\]
where $\g_{v,\gamma}$ denotes the centraliser of $\gamma$ in $\g_v$.  If $S$ is any set of places and $\gamma \in G(F)$, we define $D_S(\gamma) = \prod_{v \in S} D_v(\gamma)$ and $D^S(\gamma) = \prod_{v \notin S} D_v(\gamma)$.

\subsection{Bounding volumes}\label{sec:volumes}
We again let $G$ denote a group satisfying the conditions of Theorem \ref{thm1}. Throughout the rest of this section, $A, B$ and $C$ will denote sufficiently large positive constants that may vary from line to line, and will never depend on a choice of place of $F$.

In preparation for the following result, we introduce some additional notation. Given $\kappa\geqslant 0$ and a finite set of finite places $T$ away from $D$, we write $U_T^{\leqslant\kappa}$ for the open compact subset ${\rm supp}\, \mathcal{H}_T^{\leqslant\kappa}$. Furthermore, we denote by $\cC_T^{\leqslant\kappa}$ the set of $G(F)$-conjugacy classes of elements in $G(F)-Z(F)$ whose $G(\A)$-conjugacy classes have non empty intersection with $K^T\cdot U_T^{\leqslant\kappa}\cdot U_\infty$, where $U_\infty = \{g\in G(F_\infty): d(g,H_{v_0}) <1\}$.

\begin{prop}\label{vol-est} There exist $A,B>0$ such that for any $\kappa\geqslant 0$, any finite set of finite places $T$ away from $D$, and any $\{\gamma\}\in\cC_T^{\leqslant\kappa}$, we have
\[
\mu^\textnormal{can}_{I_\gamma}(I_\gamma(F)\backslash I_\gamma(\A)) \ll q_T^{A\kappa + B}.
\]
The implied constant depends only on $G$.
\end{prop}

\begin{proof}
Let $S_D$ denote the set of places dividing $D$. Put $S_\gamma = \{ v \notin S_D\cup \infty \cup T : D_v(\gamma) \neq 1 \}$. We begin by noting that for any $\{\gamma\} \in \cC_T^{\leqslant\kappa}$ we have
\be\label{Igammavol}
\mu^\textup{can}_{I_\gamma}( I_\gamma(F)\backslash I_\gamma(\A)) \ll q_{S_\gamma}^B q_T^B,
\ee
where the implied constant depends only on $G$. Indeed, from the proof of \cite[Corollary 6.16]{ST} we have
\bes
\epsilon(M_{I_\gamma})L(M_{I_\gamma}^\vee(1)) \ll \prod_{v\in\text{Ram}(I_\gamma)}q_v^B\ll q_{S_\gamma}^B q_T^B,
\ees
where $\text{Ram}(I_\gamma)$ is the set of finite places where $I_\gamma$ is ramified. Note that the last bound follows from the inclusion $\text{Ram}(I_\gamma) \subset S_D\cup S_\gamma \cup T$, which follows from \cite[Proposition 7.1]{Kott}. Moreover, from the definition of the local archimedean factors in (7.1) and (7.2) of \cite{Gr}, combined with \cite[Proposition 6.3]{ST}, we find that $L_\infty(M_{I_\gamma}^\vee (1))\ll 1$, the implied constant depending only on $G$. Finally, Corollary 8.12 and Lemma 8.13 of \cite{ST} imply $\mu^\textup{Tam}_{I_\gamma}(I_\gamma(F)\backslash I_\gamma(\A))\ll 1$. By combining these estimates with \eqref{cantam} we obtain \eqref{Igammavol}. 

 Now for $\{\gamma\}\in\mathcal{C}^{\leqslant\kappa}_T$ we have
\[
D_v(\gamma)\leqslant
\begin{cases}
 q_v^{A\kappa + B},&\text{ for } v \in T,\\
C,&\text{ for } v\mid \infty,\\
 1,&\text{ for } v \notin T\cup \infty.
\end{cases}
\]
From this and the product formula we deduce that
\begin{equation}\label{prod-form}
1=\prod_{v\in T}D_v(\gamma)\prod_{v\in S_D\cup\infty}D_v(\gamma)\prod_{v\in S_\gamma}D_v(\gamma)\ll q_T^{A\kappa +B}q_{S_\gamma}^{-1},
\end{equation}
since $D_v(\gamma)\leqslant q_v^{-1}$ for every $v\in S_\gamma$. Inserting this into \eqref{Igammavol} gives the proposition.
\end{proof}

\subsection{Bounding adelic orbital integrals}\label{adelicOI}

We shall prove Theorem \ref{trace-guy} by inserting test functions of the form $\phi = {\bf 1}_N^T\otimes k_T\otimes k_\xi$ into the trace formula, where $k_T \in \cH_T$, $k_\xi \in \cH_\infty$, and ${\bf 1}_N^T$ was defined in \eqref{def-eps}.  We now estimate the orbital integrals of these functions.

\begin{prop}\label{propOI}
There are constants $A,B,\delta,\eta >0$ such that the following holds. Let $T$ be a finite set of finite places away from $N$ and $D'$. Let $\xi\in\mathfrak{a}^*$ and $\kappa\geqslant 0$. Then for $k_T\in\mathcal{H}_T^{\leqslant \kappa}$, any $k_\xi\in\mathcal{H}_\infty$ satisfying $({\mathsf P}_\infty)$, and any $\gamma\in G(F)-Z(F)$, we have
\[
O_\gamma({\bf 1}_N^T\otimes k_T\otimes k_\xi) \ll  N^{-\delta}\beta(\xi) (1 + \|\xi\|)^{-\eta} q_T^{A\kappa + B}\| k_T\|_\infty,
\]
where the implied constant depends only on $G$, $\theta$, and $K$.
\end{prop}

\smallskip

\noindent {\it Proof.} We may write $k_T$ as a linear combination of at most $q_T^{A\kappa + B}$ terms of the form $\otimes_{v \in T} 1_{K_v \mu_v(\varpi_v) K_v} \in \cH_T^{\leqslant \kappa}$, whose sup norms are all bounded by $\| k_T\|_\infty$.  We may therefore assume that $k_T$ is a multiple of $\otimes_{v \in T} 1_{K_v \mu_v(\varpi_v) K_v}$.  This assumption implies that the orbital integral factorizes as $\prod_vO_\gamma(\phi_v)$, where for any $\gamma_v\in G(F_v)$ we have
\[
O_{\gamma_v}(\phi_v)=\int_{I_{\gamma_v}(F_v)\backslash G(F_v)}\phi_v(x_v^{-1}\gamma_v x_v)d\mu_{\gamma,v}(x_v)
\]
and $\mu_{\gamma,v} = \mu^\text{can}_{G,v} / \mu^\text{can}_{I_{\gamma_v},v}$. It therefore suffices to work place by place.

In \cite[Theorem 14.1]{ST} it is shown that, if we choose $D'$ to be a sufficiently large (ineffective) multiple of $D$ depending only on $G$, then
\be\label{OIforT}
|O_\gamma (k_T)| \leqslant q_T^{A\kappa+B} D_T(\gamma)^{-1/2}\| k_T\|_\infty.
\ee
We may prove the following bound for the integral at infinity using the results of Section \ref{ROI}.

\begin{lemma}\label{lemma-real}
For any $0 < \eta < 1/2$, we have the bound
\[
O_\gamma(k_\xi)\ll \beta(\xi) (1 + \|\xi\|)^{-\eta}D_\infty(\gamma)^{-3/4}.
\]
\end{lemma}

\begin{proof}
Let $G_{\rm cpt}$ and $G_{{\rm cpt},v_0}$ be the groups associated with $G_\infty$ and $G_{v_0}$ by Definition \ref{def:Gcpt}. We begin by showing that, as a consequence of Proposition \ref{blowupint}, the following statement holds. Let $0 < \eta < 1/2$, and let $k_\xi\in\mathcal{H}_\infty$ satisfy property $(\mathsf{P}_\infty)$. Then
\be\label{deduce-O-phi}
O_\gamma(k_\xi)\ll \beta(\xi) (1 + \|\xi\|)^{-\eta}D_\infty(\gamma)^{-3/4}
\ee
for every semisimple $\gamma\in G_\infty- G_{\rm cpt}$.

To see how \eqref{deduce-O-phi} follows from Proposition \ref{blowupint}, first note that for any non-negative $f\in C_c^\infty(G_\infty^0)$ such that $f =1$ on the support of $k_\xi$, we have, using property $(\mathsf{P}_\infty)$, 
\begin{align*}
k_\xi(g) & \ll \beta(\xi) ( 1 + \|\xi\| d(g,H_{v_0}))^{-1/2} f(g) \\
& \leqslant \beta(\xi) ( 1 + \|\xi\| d(g,H_{v_0}))^{-\eta} f(g) \\
& \ll \beta(\xi) (1 + \| \xi \|)^{-\eta} d(g,H_{v_0})^{-\eta} f(g).
\end{align*}
Thus
\[
O_\gamma(k_\xi)\ll \beta(\xi) (1 + \|\xi\|)^{-\eta} O_\gamma(f d( \cdot ,H_{v_0})^{-\eta}),
\]
to which we may apply \eqref{prove-O-f}. Indeed, since $Z(G)$ is finite, the function $\| X(g) \|_0$ used there satisfies $\| X(g) \|_0 \ll d(g,H_{v_0})$.

The proof of Lemma \ref{lemma-real} then follows from \eqref{deduce-O-phi} once we have verified that a non-central element $\gamma\in G(F)$ cannot lie in $G_{\rm cpt}$.  Because $G_\infty = \prod_{v | \infty} G_v$, we have $G_{\rm cpt} = G_{{\rm cpt},v_0} \times \prod_{v \neq v_0} G_v$, and so it suffices to verify that $\gamma_{v_0} \notin G_{{\rm cpt},v_0}$.

In the case at hand, $G_{{\rm cpt},v_0}$ is a normal subgroup of $G_{v_0}$ which, as $G$ is semisimple, is compact. If we let $H^+$ be the fixed point set of $\theta$ in $G$ (we write $H^+$ to distinguish it from its identity component $H$), then $H^+_{v_0}$ is a maximal compact subgroup of $G_{v_0}$, and so we have $G_{{\rm cpt},v_0} \subset g H^+_{v_0} g^{-1}$ for all $g \in G_{v_0}$.  Thus, if $\gamma_{v_0} \in G_{{\rm cpt},v_0}$, we have $\gamma \in g H^+(F) g^{-1}$ for all $g \in G(F)$.  The group
\[
\bigcap_{g \in G(F)} g H^+ g^{-1}
\]
is a proper $F$-subgroup of $G$; it is normal, since its normalizer contains the Zariski-dense set $G(F)$.  Our assumption that $G$ is $F$-almost simple implies that it must be contained in $Z(G)$.  We therefore have $\gamma \in Z(F)$, a contradiction.
\end{proof}

It remains then to address the size of the orbital integral at finite places away from $T$. This is provided by the following result. We retain the set-up from the statement of Proposition \ref{propOI}; in particular, we recall the notation ${\bf 1}_N^T$ from \eqref{def-eps}.

\begin{lemma}\label{OInotT}
There are constants $\delta, A,B,C>0$ such that for $\{\gamma\}\in \mathcal{C}_T^{\leqslant\kappa}$ we have
\begin{equation}\label{not-T}
O_\gamma({\bf 1}_N^T) \ll N^{-\delta}q_T^{A\kappa + B} D^{T\infty}(\gamma)^{-C}.
\end{equation}
\end{lemma}

\begin{proof}
Recall the sets $S_D$ and $S_\gamma$ from the proof of Proposition \ref{vol-est}. Let $S_N$ denote the set of places dividing $N$.

\medskip

\noindent $\bullet$ If $v \notin S_D \cup S_N \cup S_\gamma \cup T \cup \infty$, then $K_v(N)K_{H_v}=K_v$ and we have $O_\gamma({\bf 1}_{K_v}) = 1$; see, for example, \cite[Corollary 7.3]{Kott}.

\medskip

\noindent $\bullet$ If $v \in S_D$, then $K_v(N)K_{H_v}=K_vK_{H_v}$ and a general bound of Kottwitz \cite[Theorem 13.1]{ST} establishes that $O_\gamma({\bf 1}_{K_vK_{H_v}}) \ll_v D_v(\gamma)^{-1/2}$.

\medskip

\noindent $\bullet$ If $v \in S_N\cup S_\gamma$ we argue as follows.

We begin by estimating the orbital integrals at places $v\in S_N$, for which we will use as a critical input the work of Finis--Lapid \cite{FL}. As the setting of \cite{FL} is that of $\Q$-groups, we shall need to restrict scalars from $F$ to $\Q$ to properly invoke their results. We thus let $p$ denote the rational prime over which $v$ lies, and we note that all places lying over $p$ belong to $S_N$. We set $K_p=\prod_{v\mid p}K_v$, and $K_p(N) K_{H,p}=\prod_{v\mid p}K_v(N) K_{H,v}$.  Factorize $N=\prod_{p|N} p^{n_p}$ and put $N_p=p^{n_p}$.

Let $\mu_{G,p}^{\rm can}$ (resp., $\mu_{I_\gamma,p}^\text{can}$) be the product measure on $G_p=\prod_{v\mid p}G_v$ (resp., $I_{\gamma,p}=\prod_{v\mid p}I_{\gamma,v}$), and let $\mu_{\gamma,p}$ be the quotient measure.  Let $\Ad : I_{\gamma,p} \backslash G_p \to G_p$ be the map $g \mapsto g^{-1} \gamma g$, which is a proper map as $\gamma$ is semisimple. By definition,
\[
O_\gamma({\bf 1}_{K_p(N)K_{H_p}})=\mu_{\gamma,p} (\Ad^{-1}( K_p(N)K_{H_p}) ).
\]
As $\Ad^{-1}(K_p)$ is compact and right-invariant under $K_p$, we may break it up into a finite number of orbits $\coprod \overline{g}_i K_p$, where $g_i \in G_p$ and $\overline{g}_i$ denotes its image in $I_{\gamma,p} \backslash G_p$.  This gives
\[
O_\gamma({\bf 1}_{K_p(N)K_{H_p}})=\sum_i \mu_{\gamma,p}(\overline{g}_i K_p \cap \Ad^{-1}( K_p(N)K_{H_p}) ).
\]
Let $g \in G_p$ be such that $x := g^{-1} \gamma g \in K_p$.  The fibers of the reduction map $g K_p \to \overline{g} K_p$ are the left orbits of $I_{\gamma,p}\cap g K_p g^{-1}$, so for any $K_p' \subset K_p$ we have
\begin{align*}
\mu_{\gamma,p}(\overline{g} K_p \cap \Ad^{-1}( K_p') ) & = \frac{\mu_{G,p}^{\rm can}(h \in gK_p: h^{-1} \gamma h \in K_p')}{\mu_{I_\gamma,p}^{\rm can}(I_{\gamma,p}\cap g K_p g^{-1})} \\
& = \frac{\mu_{G,p}^{\rm can}(k\in K_p: k^{-1} x k\in K_p')}{\mu_{I_\gamma,p}^{\rm can}(I_{\gamma,p}\cap g K_p g^{-1})}.
\end{align*}
Taking the quotient of these with $K_p'$ equal to $K_p$ and $K_p(N)K_{H_p}$, and using the fact that $\mu_{G,p}^{\rm can}(K_p) = 1$ for all $v \notin S_D \cup \infty$, gives
\[
\frac{\mu_{\gamma,p}(\overline{g} K_p \cap \Ad^{-1}( K_p(N)K_{H_p}) )}{\mu_{\gamma,p}( \overline{g} K_p )}=\mu_{G,p}^{\rm can}(k\in K_p: k^{-1} x k\in K_p(N)K_{H_p}).
\]

One can deduce from  Propositions 5.10 and 5.11 in \cite{FL} (see Section \ref{sec:F-L} below) that there are constants $\epsilon, \delta>0$ (independent of $p$) such that if $x\in K_p$ is conjugate to an element of $G(F) - Z(F)$ and satisfies $D_p(x) > N_p^{-\epsilon}$ one has
\be\label{FL-bound}
\mu_{G,p}^{\rm can}(k\in K_p: k^{-1}xk\in K_p(N) K_{H,p}) \ll N_p^{-\delta}.
\ee
(In particular, the implied constant is independent of $p$.)  On the other hand $\sum_i \mu_{\gamma,p}(\overline{g}_i K_p) = \mu_{\gamma,p}(\Ad^{-1}(K_p) ) =O_\gamma ({\bf 1}_{K_p})$. Noting that $D_p(\gamma) = D_p(g_i^{-1} \gamma g_i)$ for all $i$, we deduce that for $D_p(\gamma)>N_p^{-\epsilon}$ we have
\be
\label{FL-bound2}
O_\gamma({\bf 1}_{K_p(N)K_{H_p}}) \ll N_p^{-\delta} \sum_i \mu_{\gamma,p}( \overline{g}_i K_p ) = N_p^{-\delta} O_\gamma ({\bf 1}_{K_p}).
\ee
In the remaining range $D_p(\gamma) \leqslant  N_p^{-\epsilon}$, we have
\[
O_\gamma({\bf 1}_{K_p(N)K_{H_p}}) \leqslant O_\gamma({\bf 1}_{K_p}) \leqslant N_p^{-\delta} D_p(\gamma)^{-\delta/\epsilon}O_\gamma({\bf 1}_{K_p}).
\]
Since $D_v(\gamma)\leqslant 1$ for all $v\notin T\cup\infty$, we may combine these as
\[
\prod_{v\in S_N}O_\gamma({\bf 1}_{K_v(N)K_{H_v}}) \ll N^{-\delta} D_{S_N}(\gamma)^{-C} \prod_{v\in S_N}O_\gamma({\bf 1}_{K_v}).
\]
Note that we may shrink $\delta$ to absorb the implied constant in \eqref{FL-bound2} for $p$ sufficiently large, so that the implied constant above only depends on $G$ and not the number of factors of $N$.

We now return to the product of orbital integrals over all $v\in S_N\cup S_\gamma$.  Recalling that $O_\gamma({\bf 1}_{K_v})=1$ for $v\notin S_\gamma$, we have just shown
\begin{align*}
\prod_{v\in S_N}O_\gamma({\bf 1}_{K_v(N)K_{H_v}})\prod_{\substack{v \in S_\gamma\\ v\notin S_N}} O_\gamma({\bf 1}_{K_v}) & \ll N^{-\delta}D_{S_N}(\gamma)^{-C} \prod_{v\in S_N}O_\gamma({\bf 1}_{K_v}) \prod_{\substack{v \in S_\gamma\\ v\notin S_N}} O_\gamma({\bf 1}_{K_v}) \\
& = N^{-\delta}D_{S_N}(\gamma)^{-C}\prod_{v\in S_\gamma}O_\gamma({\bf 1}_{K_v}).
\end{align*}
For $v \in S_\gamma$, we proceed as follows.  If $v \nmid D'$ we again apply \cite[Theorem 14.1]{ST} to get $O_\gamma({\bf 1}_{K_v}) \leqslant q_v^B D_v(\gamma)^{-1/2}$, and if $v \mid D'$ we apply \cite[Theorem 13.1]{ST} to get $O_\gamma({\bf 1}_{K_v}) \ll_v D_v(\gamma)^{-1/2}$. Combining these gives
\bes
\prod_{v\in S_\gamma}O_\gamma({\bf 1}_{K_v}) \leqslant q_{S_\gamma}^B D_{S_\gamma}(\gamma)^{-C}.
\ees
Since $\{\gamma\}\in\mathcal{C}^{\leqslant\kappa}_T$ we may invoke \eqref{prod-form} to obtain
\bes
\prod_{v\in S_\gamma}O_\gamma({\bf 1}_{K_v}) \leqslant q_{S_\gamma}^B D_{S_\gamma}(\gamma)^{-C}  \ll q_T^{A\kappa + B} D_{S_\gamma}(\gamma)^{-C}.
\ees
Putting these estimates together completes the proof of Lemma \ref{OInotT}.
\end{proof}

\begin{remark}\label{rmk:F-L}
As the authors point out in \cite[Remark 7.4]{ST}, the bound \cite[Theorem 14.1]{ST} is uniform in the place $v \nmid D'$ whereas the bound \cite[Theorem 13.1]{ST} of Kottwitz applies to $v\mid D'$ but is not uniform in $v$. As we allow the implied constant in Proposition \ref{propOI} to depend on the group, this non-uniformity is not an issue.
\end{remark}

Taken together (and using the product rule for the product of Weyl discriminants), the estimates in \eqref{OIforT}, Lemma \ref{lemma-real}, and Lemma \ref{OInotT} imply Proposition \ref{propOI}.\qed

\subsection{The work of Finis--Lapid}\label{sec:F-L}
We now explain how to extract from Propositions 5.10 and 5.11 of \cite{FL} the bound we stated in \eqref{FL-bound}. Recall the notation $G_p=\prod_{v\mid p}G_v$, $K_p=\prod_{v\mid p}K_v$, and so on, from the proof of Lemma \ref{OInotT}. In what follows, we simplify notation by writing $K_p'$ for $K_p(N) K_{H,p}$. We recall that $N=\prod_{p\mid N}p^{n_p}$ and that $x$ is taken to lie in $K_p$.

We first remark that we may assume that $x$ lies in $K_p'$, for if there is no such representative then the left-hand side of \eqref{FL-bound} is zero. We then have
\[
\mu_{G,p}^{\rm can}(k\in K_p: k^{-1}xk\in K_p')=\mu_{G,p}^{\rm can}(k\in K_p: [k,x]\in K_p'),
\]
where $[k,x] := k x k^{-1} x^{-1}$ is the commutator.  The bounds of Finis--Lapid are stated in terms of the adjoint group $G^\text{ad}$. They choose an embedding $\rho_\Q^\text{ad} : \text{Res}_{F/\Q} G^\text{ad} \to\GL(N^\text{ad})$, and for every $p$ define $K_p^\text{ad} = (\rho_\Q^\text{ad})^{-1}(\GL( N^\text{ad}, \Z_p ) )$. In \cite[Definition 5.1]{FL}, they define the function
\[
\phi_{K_p'}(x)={\rm vol}(k\in K_p^{\rm ad}: [k,x]\in K_p')
\]
for $x\in K_p^{\rm ad}$, where ${\rm vol}$ is the probability Haar measure on $K_p^{\rm ad}$.\footnote{Note that Finis--Lapid consider the commutator as a map $G^\text{ad} \times G^\text{ad} \to G$ in this definition.} We first relate the two quantities.
\begin{lemma}
Let $\pi:G \rightarrow G^{\rm ad}$ be the natural projection map. Then for any $x\in K_p$,
\[
\mu_{G,p}^{\rm can}(k\in K_p: [k,x]\in K_p') \ll \phi_{K_p'}(\pi(x)),
\]
the implied constant depending only on $G$.\end{lemma}
\begin{proof}
We first claim that $\pi(K_p) \subset K_p^\text{ad}$ for almost all $p$. To see this, we choose a $\Q$-embedding $\rho_\Q: \text{Res}_{F/\Q} G \to \GL(N')$ with $K_p = \rho_\Q^{-1}( \GL(N',\Z_p) )$ for almost all $p$. The embedding
\[
\Delta : \text{Res}_{F/\Q} G \xrightarrow{\text{id} \times \pi} \text{Res}_{F/\Q} G \times \text{Res}_{F/\Q} G^\text{ad} \xrightarrow{\rho_\Q \times \rho_\Q^\text{ad} } \GL(N' + N^\text{ad})
\]
is such that $\Delta^{-1}( \GL(N' + N^\text{ad}, \Z_p) ) \subset K_p$ is hyperspecial for almost all $p$, whence the claim.  Increasing $D$ if necessary, we may then assume that $\pi(K_p) \subset K_p^\text{ad}$ for all $p \nmid D$.

For convenience, let us write
\[
U=\{k\in K_p: [k,x]\in K_p'\}\quad\text{and}\quad V=\{k\in K_p^{\rm ad}: [k,\pi(x)]\in K_p'\}.
\]
We want to show that $\mu_{G,p}^{\rm can}(U)\ll {\rm vol}(V)$. From the inclusion $\pi(K_p) \subset K_p^\text{ad}$ it follows that $\pi(U)$ is contained in both $\pi(K_p)$ and $V$, yielding
\[
\mu_{G,p}^{\rm can}(U)\leqslant \pi_*\mu_{G,p}^{\rm can}(\pi(U))\leqslant  \pi_*\mu_{G,p}^{\rm can}( V\cap \pi(K_p)).
\]
Now, on $\pi(K_p)$ the push-forward measure $\pi_*\mu_{G,p}^{\rm can}$ is just $[K^{\rm ad}_p:\pi(K_p)]{\rm vol}$. Indeed, since $K_p$ is maximal compact we have $\pi^{-1}(K_p^{\rm ad})=K_p$ so that $\pi_*\mu_{G,p}^{\rm can}(\pi(K_p))=\mu_{G,p}^{\rm can}(K_p)=1$, while ${\rm vol}(\pi(K_p))=[K^{\rm ad}_p:\pi(K_p)]^{-1}$. We deduce that
\[
\pi_*\mu_{G,p}^{\rm can}( V\cap \pi(K_p))=[K^{\rm ad}_p:\pi(K_p)]{\rm vol}(V\cap \pi(K_p))\leqslant [K^{\rm ad}_p:\pi(K_p)]{\rm vol}(V).
\]
Finally, $\pi(K_p)=K_p^{\rm ad}\cap \pi(G_p)$ so that $[K^{\rm ad}_p:\pi(K_p)]\leqslant [G^{\rm ad}_p:\pi(G_p)]$, which is bounded in terms of $G$.
\end{proof}

We now fix a $\Z$-lattice $\Lambda$ in ${\rm Res}_{F/\Q}(\g)$ such that $\Lambda\otimes\widehat{\Z}$ is $K^{\rm ad}$-stable. For $y\in G^{\rm ad}_p$ define $\lambda_p(y)$ as in \cite[Definition 5.2]{FL}, namely, 
\[
\lambda_p(y)=\max\{n\in\Z\cup\{\infty\}: (\Ad(y)-1){\rm Pr}_\mathfrak{h}(\Lambda\otimes\Z_p)\subset p^n(\Lambda\otimes\Z_p)\text{ for some }\mathfrak{h}\neq 0\},
\]
where $\mathfrak{h}$ ranges over the nontrivial $\Q_p$-ideals of $\g_p=\prod_{v\mid p}\g_v$ and ${\rm Pr}_\mathfrak{h}$ denotes the corresponding projection $\g_p\rightarrow\mathfrak{h}\subset\g_p$. 

We first show how $\lambda_p(y)$ controls the size of $\phi_{K_p'}(y)$, using a combination of Propositions 5.10 and 5.11\footnote{Note that Propositions 5.10 and 5.11 of \cite{FL} assume that $G$ is simply connected. This assumption can be dropped for those subgroups not containing the intersection of $K_p$ with $G_p^+$. Here, $G_p^+$ denotes the image in $G_p$ of the $\Q_p$-points of the simply connected cover of $G_p$. That our subgroups $K_p'=K_p(N) K_{H,p}$ satisfy this condition (for $p$ large enough with respect to $G$ and $F$) can be seen from comparing indices. The subgroups $K_p(N) K_{H,p}$ have indices growing like a power of $p$, whereas those containing $K_p\cap G_p^+$ are of index bounded in terms of $G$ and $F$.} of \cite{FL}. The argument is already present in \cite[Section 5.2]{FL} in the deduction of the global result \cite[Theorem 5.3]{FL} from these two local results.

\begin{lemma}
For every $\epsilon>0$ small enough there is $\delta>0$ such that whenever $y\in K_p^{\rm ad}$ verifies $\lambda_p(y)<\epsilon n_p$ then $\phi_{K_p'}(y)\ll_G p^{-\delta n_p}$. 
\end{lemma}

\begin{proof}
From \cite[Proposition 5.11]{FL} it follows that there are positive constants $a,b>0$ and $c\geqslant 0$ (depending only on $G$) such that $\phi_{K_p'}(y)\leqslant p^{a(c+\lambda_p(x)-b n_p)}$. The constant $c$ renders this bound useless for small $n_p$. Taking $\epsilon$ small enough to satisfy $0<\epsilon< (c+1)^{-1}b$, we shall apply this bound only in the range $n_p>\epsilon^{-1}$; we obtain $\phi_{K_p'}(y)\leqslant p^{-a(b-\epsilon (c+1))n_p}$. In the remaining range $n_p\leqslant \epsilon^{-1}$ we see that $\lambda_p(y)<1$ so that in fact $\lambda_p(y)=0$. In this case \cite[Proposition 5.10]{FL} ensures that $\phi_{K_p'}(y)\ll_G p^{-1}\leqslant p^{-\epsilon n_p}$. Taking $\delta=\min\{\epsilon, a(b-\epsilon (c+1))\}$, we establish the claim.
\end{proof}

It remains to understand the relation between $D_p(y)$ and $\lambda_p(y)$, for $y\in K_p^{\rm ad}$. For this, the following lemma will be helpful.

\begin{lemma}
\label{Adnontriv}
If $\gamma\in G(F)-Z(F)$, $\Ad(\gamma)$ acts nontrivially on every nontrivial $\Q_p$-ideal of $\g_p$.

\end{lemma}

\begin{proof}

Suppose that $\gh \subset \g_p$ is a nontrivial $\Q_p$-ideal on which $\Ad(\gamma)$ acts trivially.  Let $v | p$ be a place for which ${\rm Pr}_{\g_v} \gh \neq 0$, and let $\gh_v$ be the $F_v$-linear span of ${\rm Pr}_{\g_v} \gh$.  Then $\gh_v$ is a nontrivial ideal of $\g_v$ on which $\Ad(\gamma)$ acts trivially, or equivalently such that $\gh_v \subset \g_{v,\gamma}$.  Since $G$ is $F$-almost simple we have $\bigcap_{g \in G(F)} g G_\gamma g^{-1} \subset Z(G)$ as in Lemma \ref{lemma-real}. This implies that $\bigcap_{g \in G(F)} \Ad(g) \g_\gamma = 0$, and hence $\bigcap_{g \in G(F)} \Ad(g) \g_{v, \gamma} = 0$.  However, this is a contradiction as $\Ad(g) \gh_v = \gh_v$ for all $g \in G_v$.
\end{proof}

Finally, we complete the proof of \eqref{FL-bound}. Note that we have assumed the $x \in K_p$ to which we apply \eqref{FL-bound} are $G_p$-conjugate to an element in $G(F)-Z(F)$. We may therefore apply the conclusion of Lemma \ref{Adnontriv} to $x$, and thus to $y=\pi(x)\in K_p^{\rm ad}$.

\begin{lemma}
Let $y\in K_p^{\rm ad}$ be semisimple. Assume that $\Ad(y)$ is nontrivial on every nontrivial $\Q_p$-ideal of $\g_p$ and $D_p(y)> p^{-\epsilon n_p}$. Then $\lambda_p(y)<\epsilon n_p$.

\end{lemma}

\begin{proof}

Let $l = \lceil \epsilon n_p \rceil$.  Now if $\lambda_p(y)\geqslant \epsilon n_p$ then, by definition, there is some nontrivial $\Q_p$-ideal $\gh$ of $\g_p$ such that
\[
(\Ad(y) - 1) \text{Pr}_\gh(\Lambda\otimes\Z_p) \subset p^{l}( \Lambda\otimes\Z_p).
\]
As $\Ad(y)$ preserves $\gh$, this implies that
\[
(\Ad(y) - 1) \text{Pr}_\gh(\Lambda\otimes\Z_p) \subset p^{l} \text{Pr}_\gh(\Lambda\otimes\Z_p).
\]
This implies that all eigenvalues of $\Ad(y) - 1$, considered as an endomorphism of the $\Q_p$ vector space $\gh$, must have $p$-adic valuation at least $l$. Moreover, as $y$ is semisimple and $\Ad(y)$ is nontrivial on $\gh$, one of these eigenvalues must be nonzero.  As $D_p(y) = | \det( 1 - \Ad(y))_{\g_p / \g_{p,y} } |_p$, we therefore have $D_p(y)\leqslant p^{-\epsilon n_p}$, a contradiction.
\end{proof}

\subsection{Proof of Theorem \ref{trace-guy}}
We recall the test function $\phi$ from the statement of Proposition \ref{propOI}, and the set $\mathcal{C}^{\leqslant\kappa}_T$ from Section \ref{sec:volumes}.  Then
\[
I_\text{spec}(\phi, \mu^\text{can}_G) = \mathrm{Vol}_N\sum_{i\geqslant 0} \widehat{k_T\otimes k_\xi}(\psi_i),
\]
the sum ranging over an orthonormal basis of Hecke--Maass forms for $Y_N$, and
\bes
I_\text{geom}(\phi, \mu^\text{can}_G) =\mu^\text{can}_G([G])  k_T(1)k_\xi(1)+\sum_{\{ \gamma \}\in \cC^{\leqslant \kappa}_T} \frac{\mu^\text{can}_{I_\gamma}(I_\gamma(F)\backslash I_\gamma(\A))}{ | G_\gamma(F) : I_\gamma(F)|} O_\gamma(\phi).
\ees
Here we have used the hypothesis on $K$ from Section \ref{cmp-sg} that $Z(F)\cap K=\{e\}$. Now by \cite[Corollary 8.10]{ST} we have $|\cC^{\leqslant \kappa}_T| \ll q_T^{A\kappa + B}$. From this and Propositions \ref{vol-est} and \ref{propOI} we find
\[
I_\text{geom}(\phi, \mu^\text{can}_G)= \mu^\text{can}_G([G])k_T(1)k_\xi(1)+O(q_T^{A\kappa+B}N^{-\delta}\beta(\xi) (1 + \|\xi\|)^{-\eta}\|k_T\|_\infty),
\]
as desired. This completes the proof of Theorem \ref{trace-guy}.

\section{Bounds for real orbital integrals}\label{ROI}

The aim of this section is to prove Proposition \ref{blowupint}, which establishes the uniform bounds on real orbital integrals that were used in the proof of the global bounds of Proposition \ref{propOI}.

\subsection{Notation}
\label{ROIint}

We adopt the following notation in this section.

\begin{itemize}

\item $G$ is a connected reductive group over $\R$ with real Lie algebra $\g$.
 
\item $\theta$ is a Cartan involution of $G$.

\item $K$ is the fixed point set of $\theta$, so that $K$ is a maximal compact subgroup of $G$.

\item $\g = \p + \gk$ is the Cartan decomposition associated with $\theta$.

\item $A$ and $A_G$ are maximal split tori in $G$ and $Z(G)$ respectively.  We assume that $\theta$ acts by $-1$ on $A$ and $A_G$, see e.g. \cite[Section 24.C]{Bo}.

\item $A^0$ is the connected component of $A$ in the real topology.

\item $\ga$ and $\ga_G$ are the Lie algebras of $A$ and $A_G$.

\item $W$ is the Weyl group of $\ga$.  We recall \cite[Section 24.C]{Bo} that this is equal to both $N_G(A) / Z_G(A)$, $N_K(A) / Z_K(A)$, and the group generated by reflections in the roots of $\ga$ in $\g$.

\item $\ga^+$ is a choice of open Weyl chamber in $\ga$.

\item $\langle \cdot, \cdot \rangle$ is an Ad-invariant bilinear form on $\g$ that is positive on $\p$ and negative on $\gk$.

\item $\langle \cdot, \cdot \rangle_+$ is the positive definite inner product on $\g$ given by $\langle u, v \rangle_+ = -\langle \theta u, v \rangle$.  We let $\| \cdot \|$ be the associated norm.

\end{itemize}

For a semisimple element $\gamma\in G$ let $G_\gamma$ be its centralizer, and let $I_\gamma$ be the neutral component of $G_\gamma$. The Lie algebra of $G_\gamma$ is denoted by $\g_\gamma$.  By \cite[Theorem 7.39]{Kn}, $G$ has a Cartan decomposition $G = K A^0 K$, and any $g \in G$ may be written as $g = k_1 e^H k_2$ for a unique $H \in \overline{\ga}^+$.  We use this to define a map $X : G \rightarrow \overline{\ga}^+$ by $g \in K e^{X(g)} K$.  We let $D^G$ (or $D$ if there is no confusion) denote the Weyl discriminant.  A Levi subgroup or parabolic of $G$ will be called semi-standard if it contains $A$.  We will always choose the Levi of a semi-standard parabolic to be semi-standard.

\subsection{Orbital integrals of continuous functions}

We shall derive Proposition \ref{blowupint} from the following result, which bounds $O_\gamma(f)$ for semisimple $\gamma \in G$ and non-singular test functions $f \in C(G)$.  It is the archimedean analog of a result of Kottwitz \cite[Theorem 13.1]{ST} in the $p$-adic case.

\begin{prop}
\label{boundedint}
If $f \in C^\infty(G)$ is bounded and compactly supported modulo center, then we have $O_\gamma(f) \ll_f D(\gamma)^{-1/2}$ for every semisimple $\gamma \in G$.
\end{prop}

\begin{proof}
The stated inequality is proved in \cite[Theorem 2]{HCFS} for $\gamma$  regular semisimple. The extension of this bound to all semisimple $\gamma$ follows immediately from Harish-Chandra's derivative formula \cite[Lemma 23]{HCDS}, which expresses $O_\gamma(f)$ as the derivative of (a quantity closely related to) $D(\sigma)^{1/2} O_\sigma(f)$, for regular semisimple $\sigma$ in a neighborhood of $\gamma$. The idea is that the latter is bounded independently of $\gamma$, and examining the constant that appears there gives the proposition.

We now explain this idea in more detail. We recall that a maximal torus $T$ in a connected reductive group $H$ is called fundamental if the compact part of $T$ has maximum possible dimension. For any semisimple $\gamma \in G$, we may choose a fundamental maximal torus $T$ in $I_\gamma$. Note that $\gamma \in I_\gamma$, which implies that $\gamma \in T$. As there are only finitely many possiblities for $T$ up to conjugacy, we may assume that $T$ is fixed. Once $T$ is fixed there are finitely many possibilities for $I_\gamma$, so we also assume $I_\gamma$ is fixed.  Note that we need to choose $T$ fundamental in order for \eqref{DIFf} below to hold with a nonzero constant $c$.

We let $\Delta$ and $\Delta_\gamma$ be the roots of $T$ in $G$ and $I_\gamma$.  We choose a system of positive roots $\Delta^+$, and let $\Delta^+_\gamma = \Delta_\gamma \cap \Delta^+$.  We let $\rho$ be a branch of the square root of the modular character, i.e. a smooth function on $T$ defined near $\gamma$ so that $\rho^2 = \prod_{\alpha \in \Delta^+} \alpha$.  Let $G_\text{reg}$ be the set of strongly regular elements in $G$, which we recall is the set of $\sigma$ such that $G_\sigma$ is a torus, and define $T_\text{reg} = T \cap G_\text{reg}$.  Let $\Omega$ be a connected component of $T_\text{reg}$ such that $\gamma \in \overline{\Omega}$, and let $V \subset \Omega$ be the intersection of $\Omega$ with an open neighborhood of $\gamma$ in $T$.  For $\sigma \in V$, we define
\[
F_f(\sigma) = \rho(\sigma) \prod_{\alpha \in \Delta^+} (1 - \alpha(\sigma)^{-1}) \int_{T \backslash G} f(x^{-1} \sigma x) dx.
\]
For any $\alpha \in \Delta_\gamma$ let $H_\alpha \in \text{Lie}(T)$ be as defined in \cite[Section 4]{HCSR}. Define $D_I$ to be the translation-invariant differential operator on $T$ corresponding to $\prod_{\alpha \in \Delta_\gamma^+} H_\alpha$. By \cite[Lemma 23]{HCDS} (adapted to the real algebraic case), we may show that
\be
\label{DIFf}
\underset{\sigma \to \gamma}{\lim} D_I F_f(\sigma) = c \rho(\gamma) \prod_{\alpha \in \Delta^+ - \Delta_\gamma^+} (1 - \alpha(\gamma)^{-1}) \int_{I_\gamma \backslash G} f(x^{-1} \gamma x) dx,
\ee
where $c \neq 0$ depends only on the isomorphism class of $I_\gamma$.  

On the other hand, it is known \cite[Theorem 2]{HCFS} that $F_f \in C^\infty(V)$, and if $X$ is any translation-invariant differential operator on $T$, we have
\[
\underset{\sigma \in V}{\sup} | X F_f(\sigma) | \ll_{X,f} 1.
\]
Combined with \eqref{DIFf}, this gives the proposition.
\end{proof}

\subsection{Orbital integrals of singular functions}\label{sec:blowup}

We now arrive at the main result of this section.  We define $\| \cdot \|_0$ to be the seminorm on $\ga$ obtained from the Killing form, which descends to a norm on $\ga / \ga_G$.

\begin{definition}\label{def:Gcpt}
We define $G_{\rm cpt}$ to be the elements of $G$ such that $\Ad(g)$ is trivial on all the $\R$-simple factors of $\g$ of noncompact type.
\end{definition}

It is clear that $G_{\rm cpt}$ is a normal subgroup of $G$, and because its image under $\Ad$ is compact, $G_{\rm cpt}$ is compact if and only if $A_G$ is trivial.

\begin{prop}
\label{blowupint}
Suppose that $G$ is not a torus.  Let $0 < \eta < 1/2$ and let $f \in C(G)$ be bounded and compactly supported modulo center.  Then $O_\gamma( f \| X( \cdot ) \|_0^{-\eta} )$ converges absolutely, and there is a constant $c(\eta,f) > 0$ such that 
\be\label{prove-O-f}
|O_\gamma( f \| X( \cdot ) \|_0^{-\eta} )| < c(\eta,f) D(\gamma)^{-3/4}
\ee
for every semisimple $\gamma\in G- G_{\rm cpt}$.
\end{prop}

\begin{remark}
It may be seen that the bound \eqref{prove-O-f} cannot hold for all semisimple $\gamma \in G$.  For instance, if $\gamma$ lies in a compact normal subgroup of $G$ (and hence in $G_\text{cpt}$) then the conjugacy class of $\gamma$ will be contained in $K$, and the function $f \| X( \cdot ) \|_0^{-\eta}$ will be singular everywhere on the conjugacy class. 
\end{remark}

We sketch the proof of Proposition \ref{blowupint}, which will occupy the remaining subsections. In Section \ref{ROIaniso} we begin by reducing to the case where $Z(G)$ is anisotropic. This is simple; if $A_G$ is the maximal split torus in $Z(G)$, we simply push the orbital integrals forward to $G / A_G$. We then divide the proof into two cases, based on the following definition.

\begin{definition}

We say that $\gamma \in G$ is elliptic if it is semisimple and $Z(G_\gamma)$ is anisotropic.

\end{definition}

If our semisimple $\gamma$ is not elliptic, we may choose a nontrivial split torus $S \subset Z(G_\gamma)$ and define $M$ to be the centraliser of $S$ in $G$.  Because we have assumed that $Z(G)$ is anisotropic, $M$ is a proper Levi subgroup satisfying $G_\gamma \subset M$, and we may apply parabolic descent as well as Proposition \ref{boundedint}. This is executed in Section \ref{ROInonell}.

If $\gamma$ is elliptic, we may assume without loss of generality that $\gamma \in K$.  For $\epsilon > 0$, define $K(\epsilon) = \{ g \in G : \| X(g) \| < \epsilon \}$.  Proving Proposition \ref{blowupint} in this case is roughly equivalent to controlling the volume of the conjugacy class of $\gamma$ that lies inside $K(\epsilon)$, uniformly in $\epsilon$ and $D(\gamma)$.  This turns out to be equivalent to bounding the set of points in $G/K$ that are moved distance at most $\epsilon$ by the rotation $\gamma$.  Note that this set will be noncompact if $I_\gamma$ is noncompact. We do this in Section \ref{ROIell} using a geometric argument and Proposition \ref{boundedint}.

\subsection{Reduction to the case of $Z(G)$ anisotropic}
\label{ROIaniso}

Let $\overline{G} = G / A_G$.  We wish to show that if Proposition \ref{blowupint} holds for $\overline{G}$ then it holds for $G$.  We have an exact sequence $1 \rightarrow A_G \rightarrow G \rightarrow \overline{G} \rightarrow 1$ of algebraic groups over $\R$, and by Hilbert 90 this gives an exact sequence on points. If $\gamma \in G$, we denote its image in $\overline{G}$ by $\overline{\gamma}$.  We denote the connected centraliser of $\overline{\gamma}$ by $\overline{I}_{\overline{\gamma}}$.  There is an exact sequence $1 \rightarrow A_G \rightarrow I_\gamma \rightarrow \overline{I}_{\overline{\gamma}} \rightarrow 1$, and so $G \rightarrow \overline{G}$ induces a bijection $I_\gamma \backslash G \simeq \overline{I}_{\overline{\gamma}} \backslash \overline{G}$.  Moreover, under this bijection we have $\mu^{\text{can}}_G / \mu^{\text{can}}_{I_\gamma} = C \mu^{\text{can}}_{\overline{G}} / \mu^{\text{can}}_{\overline{I}_{\overline{\gamma}}}$ for some $C > 0$, and the function $\| X( \cdot) \|_0$ pushes forward to $\| X( \cdot) \|$.  It follows that if $f \in C(G)$ is invariant under $A_G$, and its reduction $\overline{f}$ lies in $C_c(\overline{G})$, we have $O^G_\gamma(f \| X( \cdot) \|_0^{-\eta}) = C \cdot O^{\overline{G}}_{\overline{\gamma}}(\overline{f} \| X( \cdot) \|^{-\eta})$.

Let $f \in C(G)$ be bounded and compactly supported modulo center. We may assume without loss of generality that $f \geqslant 0$, by replacing $f$ by its absolute value. Since $f$ is compactly supported modulo the center, we may choose $h \in C(G)$ that is invariant under $A_G$, satisfies $h \geqslant f$, and has reduction $\overline{h}$ lying in $C_c(\overline{G})$.  We have
\[
O^G_\gamma(f \| X( \cdot) \|_0^{-\eta}) \leqslant  O^G_\gamma(h \| X( \cdot) \|_0^{-\eta}) = C \cdot O^{\overline{G}}_{\overline{\gamma}}(\overline{h} \| X( \cdot) \|^{-\eta}).
\]
Because $D(\gamma) = D(\overline{\gamma})$, if Proposition \ref{blowupint} holds for $\overline{G}$ then it holds for $G$.

We may henceforth assume that $Z(G)$ is anisotropic. We may then take the function $f$ of Proposition \ref{blowupint} to lie in $C_c(G)$, and work with $\| X( \cdot) \|$ instead of $\| X( \cdot) \|_0$. As in Section \ref{sec:atf}, we write $\mu_\gamma=\mu^{\text{can}}_G / \mu^{\text{can}}_{I_\gamma}$.

\subsection{The case when $\gamma$ is not elliptic}
\label{ROInonell}

We handle the case when $\gamma$ is not elliptic by the process of parabolic descent, which we now recall.  Let $P = MN$ be a semi-standard parabolic subgroup of $G$. Choose Haar measures on $N$ and $K$ so that $dk$ gives $K$ measure 1, and for which $d\mu_G^\text{can} = d\mu_M^\text{can} dn dk$ in Langlands $MNK$ co-ordinates.  For functions in $C_c(G)$, the parabolic descent along $P$ is defined by
\bes
f \in C_c(G) \mapsto f^P \in C_c(M),
\ees
where
\bes
f^P(m) = \delta^{1/2}_P(m) \int_N \int_K f(k^{-1} mn k) dn dk.
\ees

If $\gamma \in M$, we define $D^G_M(\gamma)$ by chosing a maximal torus $\gamma \in T \subset M$, letting $\Delta$ and $\Delta_M$ be the roots of $T$ in $G$ and $M$, and setting $D^G_M(\gamma) = \prod_{\alpha \in \Delta - \Delta_M} |\alpha(\gamma) - 1|$.  It may be seen that this is independent of the choice of $T$.  We say that $\gamma \in M$ is $(G,M)$-regular if $D^G_M(\gamma) \neq 0$.  We recall the descent relation between the orbital integrals of $f$ and $f^P$.

\begin{lemma}
\label{descent}

If $\gamma \in M$ is $(G,M)$-regular and $f \in C_c(G)$, we have
\bes
D^G_M(\gamma)^{1/2} O_\gamma(f) = O_\gamma^M(f^P).
\ees

\end{lemma}

\begin{proof}

Because $\gamma$ is $(G,M)$-regular, we have $I_\gamma \subset M$.  Let $\mu_\gamma^M = \mu_M^\text{can} / \mu_{I_\gamma}^\text{can}$.  We may parametrize $I_\gamma \backslash G$ in Langlands co-ordinates as $(I_\gamma \backslash M) N K$, which allows us to write
\begin{align*}
O_\gamma(f) & = \int_{I_\gamma \backslash G} f(x^{-1} \gamma x) d\mu_\gamma(x) \\
& = \int_{I_\gamma \backslash M} \int_N \int_K f(k^{-1} n^{-1} m^{-1} \gamma m n k ) d\mu_\gamma^M(m) dn dk \\
& = D^G_M(\gamma)^{-1/2} \delta_P(\gamma)^{1/2} \int_{I_\gamma \backslash M} \int_N \int_K f(k^{-1} m^{-1} \gamma m n k ) d\mu_\gamma^M(m) dn dk \\
& = D^G_M(\gamma)^{-1/2} O_\gamma^M(f^P),
\end{align*}
as desired.
\end{proof}

We shall need a version of Lemma \ref{descent} that can be applied to the singular functions $f \| X(\cdot) \|^{-\eta}$. This is Lemma \ref{descentsing} below, which we prove by adapting methods from \cite{MT}, where similar statements were established for $A_G\backslash\GL_n(\R)$. If $A \in GL(\g)$, define ${}^t\!A$ to be the transpose with respect to $\langle \cdot, \cdot \rangle_+$.  As in \cite[$\mathsection$4.2]{MT}, we define the functions $\cL, \cN: GL(\g) \rightarrow \R$ by $\cL(A) = \log( \tr( A\, {}^t\!A) / \dim \g )$ and $\cN(A) = \tr( A\, {}^t\!A) / \dim \g$.

\begin{lemma}
\label{LX}
We have $0 \leqslant \cL(\Ad(g)) \leqslant 2\| X(g) \|$ for all $g \in G$.
\end{lemma}

\begin{proof}
We follow \cite[Lemma 4.2]{MT}. It may be seen that ${}^t\!\Ad(g) = \Ad(\theta(g))^{-1}$ for $g \in G$. We write $g \in G$ as $k_1 e^{X(g)} k_2$, so that $\Ad(g)\, {}^t\!\Ad(g) = \Ad(g \theta(g)^{-1}) = \Ad(k_1 e^{2X(g)} k_1^{-1})$. Taking traces gives
\[
\cN(\Ad(g)) = \tr( \Ad( e^{2X(g)}) )/ \dim \g \leqslant e^{2 \| X(g) \|},
\]
and taking logs gives the upper bound. One sees that $\tr( \Ad( e^{2X(g)}) ) \geqslant \dim \g$ by applying the arithmetic mean--geometric mean inequality to opposite pairs of root spaces of $\g$, which gives the lower bound.
\end{proof}

We let $P = MN$ be as above, and let $\gn$ be the Lie algebra of $N$.

\begin{lemma}
\label{Nbound}

Let $B_M \subset M$ and $B_\gn \subset \gn$ be compact.  There is $C > 0$ depending on $B_M$ and $B_\gn$ such that for all $m \in B_M$, $V \in B_\gn$, we have $\cN( \Ad(m e^V) ) \geqslant 1 + C \| V \|^2$.

\end{lemma}

\begin{proof}

Choose a basis for $\g$ subordinate to the root space decomposition $\g = Z_\g(\ga) \oplus \bigoplus_{\alpha \in \Delta} \g_\alpha$ that is orthonormal with respect to $\langle \cdot, \cdot \rangle_+$, where $\Delta$ are the roots of $\ga$ in $\g$.  It may be seen that the nonzero entries of $\Ad(m)$ and $\Ad(m)( \Ad(e^V) - 1)$ with respect to this basis are disjoint, so we have
\bes
\cN( \Ad(me^V)) = \cN( \Ad(m)) + \cN( \Ad(m)( \Ad(e^V) - 1) ) \geqslant 1 + \cN( \Ad(m)( \Ad(e^V) - 1) ).
\ees
It follows that if $\cN(\Ad(me^V)) = 1$ then $\cN( \Ad(m)( \Ad(e^V) - 1) ) = 0$, so $V = 0$. By a compactness argument, we may therefore assume that $\| V \|$ is less than an arbitrarily small constant.  We deal with small $V$ by applying the Taylor expansion of $e^{V}$, which gives
\begin{align*}
\Ad(m)( \Ad(e^{V}) - 1) & = \Ad(m)\ad(V) + O_{B_M}( \| V \|^2) \\
\Ad(m)( \Ad(e^{V}) - 1)\, {}^t[\Ad(m)( \Ad(e^{V}) - 1)] & = \Ad(m)\ad(V)\, {}^t[\Ad(m)\ad(V)] + O_{B_M}( \| V \|^3) \\
\cN( \Ad(m)( \Ad(e^{V}) - 1) ) & = \cN( \Ad(m) \ad(V) ) + O_{B_M}( \| V \|^3 ).
\end{align*}
By compactness, we have $\cN( \Ad(m) \ad(V)) \geqslant C > 0$ for all $\| V \| = 1$ and $m \in B_M$.  As $\cN$ is quadratic, we have $\cN( \Ad(m) \ad(V)) \geqslant C \| V \|^2$ for all $V$ and $m \in B_M$, which completes the proof.
\end{proof}

We next show that the parabolic descent integral of $f \| X( \cdot ) \|^{-\eta}$ along $P$ converges and defines a bounded function on $M$.

\begin{lemma}
\label{descentreg}
Let $f \in C_c(G)$, and for $0 < \eta < 1/2$ define $F_\eta = f \| X( \cdot ) \|^{-\eta}$. Then
\begin{enumerate}[(i)]

\item the integral defining $F^P_\eta(m)$ converges absolutely for any $m \in M$;

\item there exists $f_1 \in C_c(M)$ depending only on $f$ and $\eta$ such that $| F^P_\eta(m) | \leqslant  f_1(m)$ for all $m \in M$.

\end{enumerate}

\end{lemma}

\begin{proof}
We are free to replace $f$ by its absolute value, so that $f \geqslant 0$.  It follows that $F^P_\eta(m)$ is defined everywhere as a value in $[0,\infty]$.  We may also assume that $f$ is bi-$K$-invariant.  Let $B \subset G$ be a compact set containing $\text{supp}(f)$. There are compact sets $B_M \subset M$ and $B_\gn \subset \gn$ depending only on $B$ such that $m e^V \in B$ implies $m \in B_M$ and $V \in B_\gn$.  It follows that $\text{supp}(F^P_\eta) \subset B_M$.

Lemma \ref{LX} implies
\bes
F_\eta^P(m) \delta_P^{-1/2}(m) = \int_N f(mn) \| X(mn) \|^{-\eta} dn \ll \int_N f(mn) \cL( \Ad(mn) )^{-\eta} dn.
\ees
Writing $n = e^V$ and applying Lemma \ref{Nbound} gives
\bes
F_\eta^P(m) \delta_P^{-1/2}(m) \ll \int_{\gn} f(m e^V) \| V \|^{-2\eta} dV.
\ees
In view of our assumption $0<\eta<1/2$, the right-hand side is bounded by $C(B, \eta)\|f \|_\infty$.  This proves (i), and (ii) follows by combining this with $\text{supp}(F^P_\eta) \subset B_M$.
\end{proof}

Combining Lemmas \ref{descent} and \ref{descentreg} gives:

\begin{lemma}
\label{descentsing}
Let $f \in C_c(G)$, and for $0 < \eta < 1/2$ define $F_\eta = f \| X( \cdot ) \|^{-\eta}$.  If $\gamma \in M$ is $(G,M)$-regular, then the integral $O_\gamma( F_\eta )$ converges absolutely, and we have
\be
\label{descentsingeq}
D^G_M(\gamma)^{1/2} O_\gamma( F_\eta ) = O_\gamma^M(F_\eta^P).
\ee
\end{lemma}

We return now to the proof of Proposition \ref{blowupint}. Let $0 < \eta < 1/2$, and let $f \in C_c(G)$ and $F_\eta = f \| X( \cdot ) \|^{-\eta}$.  We assume that $f \geqslant 0$.  As $\gamma$ is not elliptic, we may let $M$ be a proper Levi subgroup with $G_\gamma \subset M$.  Because $G_\gamma \subset M$, $\Ad(\gamma) - 1$ must be invertible on $\g / \text{Lie}(M)$, so that $\gamma$ is $(G,M)$-regular.  Let $P$ be a parabolic with Levi $M$.  By conjugation, we may assume that $P$ and $M$ are semi-standard. We apply Lemma \ref{descentreg} to $F_\eta$ to obtain $f_1 \in C_c(M)$ such that $F^P_\eta(m)\leqslant  f_1(m)$ for all $m \in M$.  Lemma \ref{descentsing} gives 
\bes
O_\gamma^G(F_\eta) = D^G_M(\gamma)^{-1/2} O_\gamma^M(F_\eta^P) \leqslant  D^G_M(\gamma)^{-1/2} O_\gamma^M(f_1).
\ees
From this and Proposition \ref{boundedint}, applied on $M$, we obtain
\bes
O_\gamma^G(F_\eta) < c(\eta,f) D^G_M(\gamma)^{-1/2} D^M(\gamma)^{-1/2} < c'(\eta,f) D^G(\gamma)^{-1/2},
\ees
as desired.

\subsection{The case of $\gamma$ elliptic}
\label{ROIell}

We first observe that, if $\gamma \in G$ is elliptic, then it is conjugate to an element of $K$.  This is because if $Z(G_\gamma)$ is anisotropic then it is compact and therefore conjugate to a subgroup of $K$.

We shall assume $\gamma \in K-G_{\rm cpt}$, and note that for such $\gamma$ we have $D(\gamma)\ll 1$, where the implied constant depends only on $G$.  For $\epsilon > 0$ we put $K(\epsilon) = \{ g \in G : \| X(g) \| < \epsilon \}$ and write $1_{K(\epsilon)}$ for the characteristic function of $K(\epsilon)$.  Proposition \ref{blowupint} will follow from a bound for $O_\gamma(1_{K(\epsilon)})$ that is uniform in $\epsilon$ and $D(\gamma)$, as provided by the following lemma.

\begin{lemma}
\label{tubeint}
For all $\gamma \in K -  G_{\rm cpt}$ and $0< \epsilon  \leqslant 2 D(\gamma)^{1/2}$, we have $O_\gamma(1_{K(\epsilon)}) \ll \epsilon^{d_\gamma} D(\gamma)^{-d_\gamma/2}$, where $d_\gamma\geqslant 1$ is the codimension of $I_\gamma K$ in $G$.
\end{lemma}

\begin{proof}

We first convert the problem to one on the symmetric space $S = G/K$. Let $ds$ be the metric tensor on $S$ associated with the norm $\| \cdot \|$ on $\p$.  We note that the distance function $d_S$ attached to $ds$ is given by $d_S(g_1, g_2) = \| X(g_1^{-1} g_2) \|$.

We observe that $O_\gamma(1_{K(\epsilon)}) = \mu_\gamma (I_\gamma \backslash G(\epsilon))$, where $G(\epsilon) = \{ x \in G : x^{-1} \gamma x \in K(\epsilon) \}$. The set $G(\epsilon)$ is right $K$-invariant, and $G(\epsilon) / K$ is the set of points $x \in S$ such that $d_S(x, \gamma x) < \epsilon$. Moreover, $G(\epsilon)$ is left $I_\gamma$-invariant, and we will see it is roughly a tube around $I_\gamma$.  Bounding the volume of $I_\gamma \backslash G(\epsilon)$ is therefore roughly equivalent to finding the radius of this tube, in a way which we now make precise.

Because $\theta(\gamma) = \gamma$, $\g_\gamma$ is $\theta$-stable.  We may then write $\g_\gamma = \p_\gamma + \gk_\gamma$ where $\gk_\gamma = \gk \cap \g_\gamma$ and $\p_\gamma = \p \cap \g_\gamma$.  Let $\p_\gamma^\perp$ be the orthocomplement of $\p_\gamma$ in $\p$. We note that $\gamma\notin G_{\rm cpt}$ implies that $\p_\gamma \neq \p$.  Indeed, if $\p_\gamma = \p$ then $\Ad(\gamma)$ fixes $\p$, and hence $[\p,\p]$.  However, $\p + [\p,\p]$ is the product of the $\R$-simple factors of $\g$ of noncompact type (see \cite[Ch. V, Thm 1.1]{He1} and the subsequent proof). In particular, the codimension $d_\gamma$ of $I_\gamma K$ in $G$ is at least $1$.  

Let $x \in G(\epsilon)$. By \cite[Ch. VI, Thm 1.4]{He1}, we may write $x = e^{X_\gamma} e^{X^\gamma} k$ with $X_\gamma \in \p_\gamma$, $X^\gamma \in \p_\gamma^\perp$, and $k \in K$.  The condition $x^{-1} \gamma x \in K(\epsilon)$ simplifies to $e^{-X^\gamma} \gamma e^{X^\gamma} \in K(\epsilon)$, which (since $\gamma\in K)$ is equivalent to $e^{-X^\gamma} e^{\Ad(\gamma) X^\gamma} \in K(\epsilon)$.  This implies that $d_S(e^{X^\gamma}, e^{\Ad(\gamma) X^\gamma}) <\epsilon$: the element $e^{X^\gamma}$ is rotated by $\gamma$ by distance at most $\epsilon$. The Cartan--Hadamard theorem \cite[Ch. I, Thm 13.1]{He1} then implies that $\| \Ad(\gamma) X^\gamma - X^\gamma \| < \epsilon$.  Because all eigenvalues of $\Ad(\gamma) - 1$ on $\p_\gamma^\perp$ have absolute value at most 2, and they occur in complex conjugate pairs, their absolute value is at least $C D(\gamma)^{1/2}$ for $C > 0$ depending only on $G$.  It follows that $\|X^\gamma\| \leqslant C \epsilon D(\gamma)^{-1/2}$. If we let $B_{\p_\gamma^\perp}(r)$ be the ball of radius $r$ around 0 in $\p_\gamma^\perp$ with respect to $\| \cdot \|$, it follows that $G(\epsilon) \subset I_\gamma \exp (B_{\p_\gamma^\perp}(r_0))  K$ with $r_0 = C \epsilon D(\gamma)^{-1/2} \leqslant 2C$.

Let $I_\gamma^c \subset I_\gamma$ be a compact set such that $\mu_{I_\gamma}^{\rm can}(I_\gamma^c)=1$. Then from the preceding paragraph we deduce that
\bes
\mu_\gamma (I_\gamma \backslash G(\epsilon)) \leqslant  \mu_G( I_\gamma^c \exp (B_{\p_\gamma^\perp}(r_0))   K).
\ees
Let $B_\p(r)$ denote the ball of radius $r$ around $0$ in $\p$ with respect to $\|\cdot\|$.  We have
\[
I_\gamma^c \exp (B_{\p_\gamma^\perp}(r_0))K \subset I_\gamma^c \exp (B_\p(r_0))K = I_\gamma^c K \exp (B_\p(r_0)).
\]
Because $I_\gamma^c K$ is compact and contained in $I_\gamma K$, the result follows from the expression for $r_0$, provided we show that our bounds do not depend on the choice of $I_\gamma^c$.

We show that, after conjugating $\gamma$ in $G$ to another element of $K$, there are only finitely many possibilities for $I_\gamma$\footnote{We have already shown in the proof of Proposition \ref{boundedint} that there are finitely many possibilities for $I_\gamma$ after conjugating $\gamma$ in $G$, but we need this stronger statement.}. Let $T_1, \ldots, T_n$ be representatives for the conjugacy classes of maximal $\R$-tori in $G$.  If $T_i^c$ denotes the maximal compact subgroup of $T_i(\R)$, we may assume that $T_i^c \subset K$ for all $i$.  As we are free to replace $\gamma$ by any conjugate, we may assume that $\gamma \in T_i$ for some $i$, which implies that $\gamma \in T_i^c \subset K$.  There are now finitely many possibilities for $I_\gamma$.
\end{proof}

We now return to the proof of Proposition \ref{blowupint}. Choose a compact set $B$ containing $\text{supp}(f)$, so that
\bes
f \| X( \cdot) \|^{-\eta} \ll_f 1_{B} + \sum_{k=1}^\infty 2^{\eta k} 1_{K(2^{-k})}.
\ees
We deduce that
\bes
O_\gamma(f \| X( \cdot) \|^{-\eta}) \ll_f O_\gamma(1_{B}) + \sum_{k=1}^\infty 2^{\eta k} O_\gamma(1_{K(2^{-k})}).
\ees
An application of Proposition \ref{boundedint} and Lemma \ref{tubeint} yields
\bes
O_\gamma( f \| X( \cdot) \|^{-\eta})\ll_f D(\gamma)^{-1/2} + \sum_{2^{k+1} \geqslant D(\gamma)^{-1/2}} 2^{\eta k - k} D(\gamma)^{-1/2} + \sum_{2^{k+1} < D(\gamma)^{-1/2}} 2^{\eta k} D(\gamma)^{-1/2}.
\ees
Our assumption that $0 < \eta < 1/2$ implies that both geometric series are bounded by $D(\gamma)^{-3/4}$, completing the proof of Proposition \ref{blowupint}.

\appendix

\section{Distinction principles and the $C^\infty$ spectrum of symmetric varieties}\label{sym-var}

There are several methods that may be used to prove power growth of eigenfunctions on arithmetic manifolds. The original proof of Rudnick and Sarnak uses a distinction principle.  This means that, for certain period integrals, if an automorphic form $\phi$ has a nonzero period then $\phi$ is exceptional in some sense, which can mean being a transfer from a smaller group, or being nontempered.

In this section, we review the method of Rudnick--Sarnak and describe an alternative approach to proving Theorems \ref{thm1} and \ref{thm2} which retains much of the spirit of their approach. It is based on a theorem of Sakellaridis on the unramified $C^\infty$ spectrum of spherical varieties, which we apply in the case of symmetric varieties. In short, we replace the global distinction principle of Rudnick--Sarnak which implies being in the image of a functorial lift by a local distinction principle which implies having non-tempered Satake parameters. What both of these implied properties have in common is that they are rare. Counting arguments can then be used to produce power growth.

\subsection{The global distinction argument}\label{RudSar}

We illustrate the distinction argument in a special case, taken from Rudnick and Sarnak's proof \cite{RS}.

Let $Q$ be the quadratic form $Q(x) = x_1^2 + x_2^2 + x_3^2 - 7x_4^2$.  If we let $V = \{ x \in \R^4 : Q(x) = -1\}$, then $V$ is a two-sheeted hyperboloid and the upper sheet is a model for $\mathbb{H}^3$.  If we let $\Gamma$ be the intersection of ${\rm O}(Q,\Z)$ with the identity component of ${\rm O}(Q, \R)$, then $Y = \Gamma \backslash \mathbb{H}^3$ is a compact hyperbolic 3-manifold.  The distinction result that Rudnick and Sarnak prove is that if $\psi \in L^2(Y)$ is orthogonal to all theta lifts of cusp forms of weight 1 on $\Gamma_1(28)$, then $\psi( (2,1,1,1)) = 0$.  The result then follows from the local Weyl law and a counting argument.  Indeed, the local Weyl law says that the average size of $|\psi((2,1,1,1))|^2$ must be 1.  However, the number of eigenfunctions on $X$ with eigenvalue $\lambda \leqslant R$ is roughly $R^3$, while the number of theta lifts in this range is roughly $R^2$.  Because the number of nonvanishing eigenfunctions is small, their values must be large to make up the right average.

The generalisation of this principle, namely that an automorphic form on $\SO(n,1)$ that is orthogonal to theta lifts from $\SL_2$ must have vanishing $\SO(n)$ periods, was used by Donnelly \cite{Do}.  It is likely that this could be used to prove Theorem \ref{thm1} on other groups of the form $\SO(m,n)$, $\U(m,n)$, or $\Sp(m,n)$.  Another distinction principle that one could apply is due to Jacquet \cite{J} (and later refined by Feigon, Lapid, and Offen in \cite{FLO, LO}), which states that a form on $\GL(n,\C)$ with a nonvanishing $\U(n)$ period must come from quadratic base change.  See \cite{JLR} for a general discussion of these ideas.

\subsection{Symmetric varieties}\label{background}

Let $F$ be a field of characteristic 0.  A symmetric variety over $F$ is a variety $X = G/H$ where $G$ is a reductive $F$-group, $\theta$ is an involution of $G$ over $F$, and $H$ is an open $F$-subgroup of the fixed point group $G^\theta$.  We refer the reader to \cite{HW,Vu} for background on these varieties.

If $S \subset G$ is a $\theta$-stable torus, let $S^+$ and $S^-$ be the neutral component of $S^\theta$ and $\{ x \in S : \theta(x) = x^{-1} \}$ respectively, so that $S^+ \times S^- \to S$ is an isogeny.  A $\theta$-stable torus $S$ is said to be $(F,\theta)$-split if it is $F$-split and $\theta$ acts on it by inversion.  A parabolic subgroup $P$ of $G$ is said to be $\theta$-split if $P$ and $\theta P$ are opposed. In this case, $P \cap \theta(P)$ is the unique $\theta$-stable Levi of both $P$ and $\theta P$.  All maximal $(F,\theta)$-split tori, and all minimal $\theta$-split parabolic $F$-subgroups, are conjugate under $G(\overline{F})$.  We define the $\theta$-split rank of $G$ to be the common dimension of such tori.  We say that $G$ is $\theta$-split if its $\theta$-split rank is equal to its absolute rank, that is if $G$ contains a $(F,\theta)$-split maximal torus. We say that $G$ is $\theta$-quasi-split if $G$ contains a $\theta$-split Borel defined over $F$. In this case, $B\cap\theta (B)$ is a $\theta$-stable torus, which may or may not be $\theta$-split.

{\it From now on we shall assume $G$ split over $F$.}  The following lemma then implies that $G$ is $\theta$-(quasi-)split if and only if $G \times \overline{F}$ is.

\begin{lemma}\label{alg-clo}

A maximal $(F,\theta)$-split torus in $G$ is also a maximal $\theta$-split torus in $G \times \overline{F}$, and a minimal $\theta$-split $F$-parabolic in $G$ is also a minimal $\theta$-split parabolic in $G \times \overline{F}$.

\end{lemma}

\begin{proof}

Let $S$ be such a torus, and let $A$ be a maximal $F$-split torus containing $S$. As in \cite[Lemma 4.5]{HW}, let $C$, $M_1$, and $M_2$ be the central, anisotropic, and isotropic factors of $Z_G(S)$ over $F$. Because $Z_G(A) = A$, loc. cit. gives $CM_1 \subset A$ so that $M_1$ is trivial and $C$ is split. We have $S \subset C$, and $C$ is $\theta$-stable, so by the maximality of $S$ we must have $S = C^-$. As $M_2 \subset G^\theta$ by loc. cit., $\theta$ acts trivially on $Z_G(S) / C^-$ so there is no strictly larger $\theta$-split torus containing $S$ over $\overline{F}$.

Let $P$ be such a parabolic. By \cite[Lemma 2.4]{HW} we may let $A \subset P$ be a $\theta$-stable maximal $F$-split torus of $G$. Then \cite[Proposition 4.7]{HW} implies that $A^-$ is a maximal $(F,\theta)$-split torus of $G$, and $Z_G(A^-) = P \cap \theta(P)$. We have shown that $A^-$ is a maximal $\theta$-split torus in $G \times \overline{F}$, and applying loc. cit. again gives that $P$ is minimal $\theta$-split in $G \times \overline{F}$.
\end{proof}

\subsection{A result of Sakellaridis on the unramified $C^\infty$ spectrum}\label{smoothspec}

In \cite{Sk1,Sk2}, Sakellaridis describes the unramified $C^\infty$ spectrum of a spherical variety. We reprise his results below in the symmetric case, and use them to describe an alternative approach to Theorems \ref{thm1} and \ref{thm2} in the case when $G_{v_0}$ is not quasi-split. Throughout, $F$ will be a $p$-adic field and, as in Section \ref{background}, the group $G$ will be split over $F$.

Let $A$ and $B$ be the $\theta$-stable maximal split torus and Borel of $G$ chosen in Lemma \ref{Saknot}. Let $\delta$ denote the modular character of $A$ with respect to $B$.  Let $\check{A}$ be the complex dual torus of $A$, so that unramified characters of $A$ correspond to elements of $\check{A}$.  Let $W$ be the Weyl group of $A$ and $\check{A}$. Recall that the irreducible unramified representations of $G$ are in bijection with $\check{A} / W$, via the map taking $\pi$ to its Satake parameter.  We introduce the torus $A_X= A / A \cap H$ and write $\check{A}_X$ for its dual torus. We then have a map $\iota: \check{A}_X \to \check{A}$ with finite kernel.  If $\chi$ is an unramified character of $A$, let $I(\chi)$ be the corresponding unitarily normalized induced representation of $G$.

The following result of Sakellaridis \cite[Theorem 1.2.1]{Sk1} describes the unramified $H$-dis\-tin\-guished $C^\infty$ spectrum in terms of the image $\iota(\check{A}_X)$ in $\check{A}$, denoted $A_X^*$. 

\begin{theorem}[Sakellaridis]\label{smoothdist}
There exists a non-zero morphism $C^\infty_c(X) \to I(\chi)$ only if $\chi$ lies in a $W$-translate of $\delta^{-1/2} A_X^*$.
\end{theorem}

From this we may deduce the following consequence.

\begin{cor}
\label{quasi-split-temp}

If $G$ is not $\theta$-quasi-split, any irreducible unramified representation $\pi$ of $G$ that occurs as a subrepresentation of $C^\infty(X)$ must be non-tempered.

\end{cor}

To see how Corollary \ref{quasi-split-temp} follows from Theorem \ref{smoothdist}, we must convert from Sakellaridis' notation to ours. Sakellaridis defines the tori $A$ and $A_X$ using an open Borel orbit $\mathring{X} \subset X$ and its stabilizer, a parabolic subgroup called the associated parabolic.  The following lemma translates his definitions to the symmetric case.

\begin{lemma}
\label{Saknot}
Let notation be as in Section \ref{background}. In particular, $G$ is assumed split. Let $x_0=H$ denote the identity coset in $X=G/H$. We may choose

\begin{itemize}

\item a $\theta$-stable maximal $F$-split torus $A$,

\item a Borel subgroup $B$ defined over $F$, and

\item a minimal $\theta$-split parabolic $F$-subgroup $P$

\end{itemize}
such that $A \subset B \subset P$ and

\begin{itemize}

\item the orbit $B x_0$ is open,

\item $P = \textup{stab}_G(B x_0)$ is the parabolic associated with the open orbit, and

\item the tori $A$ and $A_X = A / A \cap H$ are the same as those associated with $B$ and $P$ in \cite[p. 8]{Sk2}.

\end{itemize} 

\end{lemma}

\begin{proof}

Let $P$ be a minimal $\theta$-split parabolic $F$-subgroup, and let $A \subset P$ be a $\theta$-stable maximal $F$-split torus of $G$. We may choose a Borel subgroup $A \subset B \subset P$.  By \cite[Lemma 4.8]{HW}, $Bx_0$ is open in $X$, and $P$ stabilizes this open orbit. It is also known that $P = \textup{stab}_G(B x_0)$, see e.g. \cite[p. 8]{Sk2}.  Let $L = Z_G(A^-)$ be the $\theta$-stable Levi of $P$, so that $A \subset B \cap L$.  It follows that our choices of $A$, $B$, $P$, and $L$ are compatible with those in \cite[p. 8]{Sk2} in the special case of a symmetric variety.  Sakellaridis defines $A_X = L / L \cap H$ in \cite[p. 8]{Sk2}.  As $A$ is a maximal torus in $L$ and $[L,L] \subset H$ \cite[Proposition 2(ii)]{Vu}, we have an exact sequence $1 \to A \cap H \to A \to L / L \cap H \to 1$ so $A_X \simeq A / A \cap H$, which completes the proof.
\end{proof}

We next show that if $G$ is not $\theta$-quasi-split, then $\delta^{-1/2} A_X^*$ does not intersect the maximal bounded subgroup of $\check{A}$. If there is some unitary $\chi \in \check{A}$ that lies in $\delta^{-1/2} A_X^*$, then we have $\chi \delta^{1/2} \in A^*_X$.  As $A_X^*$ is the set of characters of $A$ that arise by pullback from $A_X$, this implies that $\chi \delta^{1/2}$ is trivial on $A \cap H$, and that $\delta$ is unitary on $A \cap H$.  However, this contradicts the following lemma.

\begin{lemma}

$\delta$ is unitary on $A^+$ if and only if $G$ is $\theta$-quasi-split.

\end{lemma}

\begin{proof}

Let $P = LU$.  We have $\delta = \delta_L \delta_U$, where $\delta_L$ and $\delta_U$ are the modular characters of $B\cap L$ and $U$ respectively.  Because $P$ and $\theta(P)$ are opposed, we have $\theta(\delta_U) = \delta_U^{-1}$, so that $\delta_U^2 = 1$ on $A^+$; by positivity, $\delta_U = 1$ on $A^+$.  Therefore $\delta|_{A^+}$ is unitary if and only if $\delta_L |_{A^+}$ is.  As $\delta_L$ is trivial on $A^-$, this is equivalent to $\delta_L$ being unitary, or to $L$ being a torus.  This completes the proof.
\end{proof}

We now turn to the proof of Corollary \ref{quasi-split-temp}.  Assume that $G$ is not $\theta$-quasi-split, and let $\pi \subset C^\infty(X)$ be irreducible, unramified, and tempered.  Let $\chi \in \check{A}/W$ be the Satake parameter of $\pi$. The contragredient $\pi^\vee$ has Satake parameter $\chi^{-1}$, and we have a non-zero map $C^\infty_c(X) \to \pi^\vee$.  By \cite[Proposition 2.6]{Ca} we may choose a representative for $\chi^{-1}$ in $\check{A}$ such that $\pi^\vee \subset I(\chi^{-1})$.  We now have a non-zero map $C^\infty_c(X)\rightarrow I(\chi^{-1})$, so that by Theorem \ref{smoothdist}, $W \chi^{-1}$ must intersect $\delta^{-1/2} A_X^*$, but this contradicts our assertion above as $\chi$ is unitary.

\subsection{An alternative approach to Theorems \ref{thm1} and \ref{thm2}}\label{altproof}

We now describe how one might use Corollary \ref{quasi-split-temp} to prove asymptotic lower bounds for periods.  The argument is in the same style as that of Rudnick and Sarnak described in Section \ref{RudSar}.

We return to the global situation of Theorems \ref{thm1} and \ref{thm2}, and freely use the notation of those statements.  We note that the compact subgroup $K_\infty$ is given by $K_{v_0} = H_{v_0}$ and $K_v = G_v$ at all other infinite places.  We shall let $N = 1$ for simplicity, and write $Y = G(F) \backslash G(\A) / K_\infty K$.  The image of $H(\A)$ in $Y$ is a finite number of points, and we shall assume for simplicity that it is a single point $p$.  Let $\mathcal{P}_H : C^\infty(G(F) \backslash G(\A)) \rightarrow \C$ be the period map $f \mapsto \int_{H(F) \backslash H(\A)} f dh$. When restricted to $C^\infty(Y)$, $\mathcal{P}_H$ is just evaluation at $p$.

Note that a general connected real reductive group $G'$ with Cartan involution $\theta$ is (quasi-) split over $\R$ if and only if $G' \times \C$ is $\theta$-(quasi-)split.  If we assume that $G_{v_0}$ is not quasi-split over $\R$ then this, together with the invariance of $\theta$ (quasi-) splitness under extension of algebraically closed fields, and Lemma \ref{alg-clo}, implies that $G_v$ is not $\theta$-quasi-split at any finite $v$ at which $G$ splits.  Corollary \ref{quasi-split-temp} then gives that any unramified representation $\pi_v$ occurring in $C^\infty(G_v / H_v)$ must be non-tempered.

Let $\pi$ be a cuspidal automorphic representation of $G$, and let $\psi \in \pi$ be invariant under $K_\infty K$.  If $\mathcal{P}_H(\psi) \neq 0$, this implies that each factor $\pi_v$ admits a non-zero smooth linear functional invariant under $H_v$.  This is equivalent to the existence of an embedding $\pi_v \rightarrow C^\infty(G_v / H_v)$, so that if $v$ is finite, $G$ splits at $v$, and all data are unramified, then $\pi_v$ is non-tempered.  The strategy would then be to use the trace formula to show that the number of such $\psi$ is a power smaller than the total number of $\psi$.  Combining this with the local Weyl law would then produce asymptotic growth.

Note that in the case when $G_{v_0}$ is quasi-split but not split over $\R$, Theorem \ref{smoothdist} only implies that the Satake parameters of the distinguished $\pi$ lie in a fixed lower-dimensional subset of $\check{A}$.  It may also be possible to use this to prove a power saving for the number of such $\pi$.

\subsection{Higher dimensional periods}

One advantage of the method described in Section \ref{altproof} is that it applies equally well to periods along positive dimensional submanifolds of $Y$ arising from rational symmetric subgroups.  Note that this would produce a result of the form ``there are certain eigenfunctions whose periods are larger than the average by a power of the eigenvalue'', while determining the size of the average period is a separate problem.  It should be pointed out that the average size of a positive dimensional period should be a negative power of the eigenvalue, so even if one could improve over this one would not necessarily obtain power growth of sup norms as a result.

In comparison, the relative trace formula approach we use is more difficult in the positive dimensional case, because the analysis of the error terms becomes much more complicated. In the case of a point, one needs to bound the value of a spherical function $\varphi_\lambda$ away from its center of symmetry. In the positive dimensional case, one needs to bound the value of an oscillatory integral whose kernel is constructed from $\varphi_\lambda$, and which is taken over two copies of the submanifold in question. Moreover, the bound obtained must be uniform as the submanifolds move.

\section{Hecke operators and the $L^2$ spectrum of symmetric varieties}

We have sketched two proofs of Theorems \ref{thm1} and \ref{thm2}, the one described in Section \ref{intro:idea}, which we carry out to completion in this paper, and the one outlined in Section \ref{altproof}. In the latter, no Hecke operators appear; in the former, no distinction principle is used. What is their connection, if any?

In this section we describe how a conjecture of Sakellaridis and Venkatesh on the $L^2$ spectrum of symmetric varieties relates to the method we have used in this paper, in particular to our choices of test functions in the trace formula. As their conjecture is expressed in terms of $X$-distinguished Arthur parameters, it acts in a sense as a bridge between the two methods.

\subsection{Plancherel measures}\label{Pl-meas}

Recall from Section \ref{intro:idea} that a comparison of trace formulae reduces the problem of finding exceptional sequences of Maass forms to that of finding $\omega$ in the global Hecke algebra with the property that $\Pi_H(\omega*\omega^*)(1)$ is large relative to $(\omega* \omega^*)(1)$. Indeed, these are just the identity distributions on the geometric side of the $H$-relative trace formula and the Arthur--Selberg trace formula. The latter question may in turn be reduced to a local problem, namely that of finding an $L^2$ normalized $\tau \in \cH_v$ satisfying
\begin{equation}\label{competing-conditions}
\tau(1) = 0\qquad\text{and}\qquad \Pi_H \tau(1) \gg 1,
\end{equation}
for $v$ in a set of places having positive density. In Section \ref{intro:idea} we explained how these two conditions can be interpreted geometrically, in terms of Hecke returns to a fixed point. We would now like to interpret these same conditions spectrally.

Until Section \ref{testchoice} we let $F$ be a $p$-adic field.  We let $G$ be a split reductive group over $F$, and let $K$ be a hyperspecial maximal compact subgroup.  We assume that the Haar measure on $G$ gives $K$ measure 1.  We let $\cH$ be the spherical Hecke algebra with respect to $K$.  We let $\widehat{G}$ be the unitary dual of $G$, and let $\widehat{G}^\text{sph}$ be the spherical unitary dual with respect to $K$.

Then the first condition in \eqref{competing-conditions} can be expressed, via the Plancherel inversion formula, as
\begin{equation}\label{oscillation}
\tau(1)=\int_{\widehat{G}^{\rm sph}} \widehat{\tau}(\nu) d\mu_G^\text{sph}(\nu) = 0.
\end{equation}
In other words, $\tau$ must be such that its Satake transform $\widehat{\tau}$ is oscillatory along the support of $\mu_G^\text{sph}$, the tempered spectrum $\widehat{G}^{\rm sph,  temp}$.

A similar Plancherel inversion formula can be used to express the second condition in \eqref{competing-conditions}. Let $X=G/H$ be a symmetric variety. We recall the existence of a Plancherel measure $\mu_X$ associated with the separable Hilbert space $L^2(X)$, viewed as a $G$-representation; roughly speaking, this is a measure on $\widehat{G}$ satisfying

\bes
L^2(X) = \int_{\widehat{G}} M(\pi) \otimes \pi \; d\mu_X(\pi),
\ees
where $M(\pi)$ is some multiplicity space.  Notice that $\mu_X$ is only defined up to absolutely continuous equivalence (we shall only be concerned with its support).

We let $\Pi_H : L^1(G) \rightarrow L^1(X)$ be given by integration over $H$.  If we let $v^0 = \Pi_H({\bf 1}_K)$, there is a second measure, the spherical Plancherel measure $\mu_X^\text{sph}$, which satisfies

\[
\langle \omega \cdot v^0,v^0\rangle_{L^2(X)} = \int_{\widehat{G}^{\rm sph}} \widehat{\omega}(\nu)\, d\mu_X^\text{sph}(\nu)
\]
for all $\omega \in \cH$.  In particular, the support of $\mu_X^\text{sph}$ is contained in the support of $\mu_X$.  Note that $\omega \cdot v^0$ denotes the action of $\omega$ on $v^0$, given by
\[
\omega \cdot v^0 =\int_{G} \omega(g) (g \cdot v^0)dg.
\]
By a simple unfolding argument, one may show that for any $\omega \in \mathcal{H}$ we have
\bes
\Pi_H\omega(1) = \int_{\widehat{G}^{\rm sph}}\widehat{\omega}(\nu)\, d\mu_X^{\textnormal{sph}}(\nu),
\ees
so that $\mu_X^\text{sph}$ determines the period of $\omega$ along $H$.  Using this, we may now rephrase the second condition in \eqref{competing-conditions} as
\be
\label{local-H-period}
\Pi_H \tau(1)=\int_{\widehat{G}^{\rm sph}} \widehat{\tau}(\nu) d\mu_X^\text{sph}(\nu) \gg 1.
\ee
In other words, $\tau$ must be such that its Satake transform $\widehat{\tau}$ does not oscillate too much on the support of $\mu_X^\text{sph}$.

Having expressed the two conditions in \eqref{competing-conditions} spectrally, we see that the existence of appropriate test functions $\tau$ can be read off from the support of $\mu_X^\text{sph}$, in particular relative to the tempered spectrum. We say that $\mu_X^\text{sph}$ is \textit{tempered} if the support of $\mu_X^\text{sph}$ is contained in $\widehat{G}^{\rm sph, temp}$ and \textit{strongly tempered} if $\mu_X^\text{sph} \leqslant C \mu_G^\text{sph}$ for some $C > 0$, where $\mu_G^\text{sph}$ is the spherical Plancherel measure on $G$. With this terminology, we summarize our discussion as follows:
\begin{itemize}

\item[(ST):] if $\mu_X^\text{sph}$ is strongly tempered the oscillation of $\widehat{\tau}$ should prevent one from simultaneously achieving both \eqref{oscillation} and \eqref{local-H-period};

\item[(T):] if $\mu_X^\text{sph}$ is tempered but not strongly tempered, the existence of $\tau$ satisfying both \eqref{oscillation} and \eqref{local-H-period} depends on how singular $\mu_X^\text{sph}$ is relative to $\mu_G^\text{sph}$;

\item[(NT):] if $\mu_X^\text{sph}$ is non-tempered, the exponential growth of $\widehat{\tau}(\nu)$ away from $\widehat{G}^{\rm sph, temp}$ should allow one (barring unforseen cancellation) to ensure both conditions \eqref{oscillation} and \eqref{local-H-period}.

\end{itemize}

In the next paragraph, we shall see how recent conjectures of Sakelleridis and Venkatesh relate the tempered properties of $\mu_X$ to the weak containment properties of $H$-distinguished representations.

\subsection{The conjecture of Sakellaridis and Venkatesh}
\label{SVsect}

Recall the dual torus $\check{A}_X$ and the map $\check{A}_X \to \check{A}$, which were of critical use in describing the local distinction argument of Section \ref{sym-var}. Let $\check{G}$ denote the Langlands dual group of $G$, containing $\check{A}$ as a maximal torus. In \cite[Section 2.2]{SV}, Sakellaridis and Venkatesh define a dual group $\check{G}_X$ associated with $X$ and a homomorphism $\iota : \check{G}_X \times \SL(2,\C) \rightarrow \check{G}$ whose restriction to $\check{G}_X$ has finite kernel. (Note that this requires imposing certain conditions on $X$, which we shall ignore as this section is purely expository.) The torus $\check{A}_X$ sits in the complex reductive algebraic group $\check{G}_X$ as a maximal torus, and the map $\iota : \check{G}_X \to \check{G}$ restricts to the natural map $\check{A}_X \to \check{A}$ above.

Again under technical assumptions that we shall ignore, $\iota(\check{G}_X)$ is equal to a group constructed by Gaitsgory and Nadler; see Section 3 of \cite{SV}, in particular Section 3.2.  Moreover, in the case of symmetric varieties, the group of Gaitsgory and Nadler is equal to the group $\check{H}$ constructed by Nadler in \cite{N}.\footnote{Note that Gaitsgory and Nadler consider spherical varieties over $\C$ rather than a $p$-adic field, but we may ignore this distinction as the dual groups are only defined using root data that are independent of the field.  Likewise, Nadler works with real reductive groups, but these are equivalent to complex reductive groups with involution.  These equivalences respect rank and quasi-splitness in the natural way.}  We continue to assume that $X$ is symmetric and now recall the following facts about $\check{G}_X$ and $\iota$.

\begin{itemize}

\item The rank of $\check{G}_X$ is equal to the $\theta$-split rank of $G$.

\item $\iota(\check{G}_X) = \check{G}$ if and only if $G$ is $\theta$-split.

\item $\iota$ is trivial on the $\SL(2,\C)$ factor if and only if $G$ is $\theta$-quasi-split.

\end{itemize}

The first claim is stated in Section 1.1 and Proposition 10.6.1 of \cite{N}.  The fact that $\iota(\check{G}_X) = \check{G}$ when $G$ is $\theta$-split is also stated in Section 1.1 there, and the reverse implication follows by considering ranks.  The third claim follows from the condition that $\iota$ be a distinguished morphism, as defined in the comment before Theorem 2.2.3 in \cite[Section 2.2]{SV}.  Indeed, the group $L$ in that comment is the Levi of a minimal $\theta$-split parabolic, and $\rho_L$ is the half sum of its positive roots, so that $\iota$ is trivial on $\SL(2,\C)$ if and only if $\rho_L$ is trivial, i.e. $L$ is a torus.  Sakellaridis and Venkatesh conjecture \cite[Conj. 16.2.2]{SV} that the support of $\mu_X$ may be described in terms of the tempered dual of $\check{G}_X$ and the map $\iota$.  They define an $X$-distinguished Arthur parameter to be a commutative diagram
\[
\xymatrix{
&\check{G}_X\times\SL_2 \ar[rd]^{\iota}  \\
\mathcal{L}_F\times\SL_2 \ar[ur]^{\phi\otimes {\rm Id}}\ar[rr] &&\check{G}
}
\]
where $\mathcal{L}_F$ is the local Langlands group of $F$, and $\phi$ is a tempered Langlands parameter for $\check{G}_X$.  This naturally gives rise to an Arthur parameter for $\check{G}$.  We shall say that an Arthur parameter for $\check{G}$ is $X$-distinguished if it arises from such a diagram, and likewise for an $X$-distinguished Arthur packet.

\begin{conj}[Sakellaridis-Venkatesh]
\label{SVconj}

The support of $\mu_X$ is contained in the Fell closure of the union of the $X$-distinguished Arthur packets for $\check{G}$.

\end{conj}

Note that the unramified members of a given Arthur packet should be contained in the associated $L$-packet, and we assume this from now on without further comment. Thus the above conjecture implies, in particular, that the support of the spherical measure $\mu_X^{\rm sph}$ is contained in the Fell closure of the union of the $L$-packets associated with $X$-distinguished Arthur parameters. The latter statement has in fact been proved in \cite{Sk2} under certain combinatorial assumptions. Let us now discuss what Conjecture \ref{SVconj} implies for $\mu_X^\text{sph}$ under the assumptions that $G$ is $\theta$-split, $\theta$-quasi-split, or neither. 

\begin{itemize}

\item[(ST):] If $G$ is $\theta$-split, then $\iota(\check{G}_X) = \check{G}$ and $\iota$ is trivial on $\SL(2,\C)$.  Conjecture \ref{SVconj} then implies that $\mu_X$ is supported on the tempered dual of $G$.  In fact, it may be shown in this case that $\mu_X^\text{sph}$ is strongly tempered.

\item[(T):] If $G$ is $\theta$-quasi-split but not $\theta$-split, $\iota$ is still trivial on the $\SL(2,\C)$ factor.  This implies that $\mu_X$ is still tempered.  However, because $\text{rank}(\check{G}_X) < \text{rank}(\check{G})$, if we identify the tempered spherical dual of $G$ with a quotient of a compact torus by the Weyl group, the support of $\mu_X^\text{sph}$ will be contained in a union of lower dimensional tori.  In particular, $\mu_X^\text{sph}$ will not be strongly tempered.

\item[(NT):] If $G$ is not $\theta$-quasi-split, then an $X$-distinguished Arthur parameter has nontrivial $\SL(2,\C)$ factor. Its underlying Langlands parameter is therefore non-tempered, and so then are all members of the associated $L$-packet. It follows that if $\psi$ is $X$-distinguished with packet $\Pi_\psi$, and $\pi \in \Pi_\psi$ is spherical, then $\pi$ must be non-tempered. From Conjecture \ref{SVconj} we deduce that the same is true for any $\pi$ in the support of $\mu_X^\text{sph}$.

\end{itemize}

\subsection{Existence of test functions}
\label{testchoice}

Finally, we return to the global situation of Theorems \ref{thm1} and \ref{thm2}.  We freely use the notation of those statements, with the exception that we drop the assumption that $G_{v_0}$ is not split. Let $v$ be a finite place at which $G$ is split and all data are unramified. We let $\mu_{G,v}^\text{sph}$ and $\mu_{X,v}^\text{sph}$ denote the spherical Plancherel measures of $G_v$ and $X_v = G_v / H_v$. As before, $G_{v_0}$ is (quasi-)split over $\R$ if and only if $G_v$ is $\theta$-(quasi-)split.

Comparing the conditions from Sections \ref{Pl-meas}-\ref{SVsect}, we obtain the following consequences of Conjecture \ref{SVconj} for the existence of $\tau$ satisfying both conditions in \eqref{competing-conditions}:

\begin{itemize}

\item[(ST):] If $G_{v_0}$ is split, then $\mu_{X,v}^\text{sph}$ is strongly tempered, and such $\tau$ should not exist. 

\item[(T):] If $G_{v_0}$ is quasi-split but not split, the existence of such $\tau$ depends on how singular $\mu_{X,v}^\text{sph}$ is. 

\item[(NT):] If $G_{v_0}$ is not quasi-split, $\mu_{X,v}^\text{sph}$ has non-tempered support, and such $\tau$ should exist.

\end{itemize}

\noindent Recalling our remarks at the end of Section \ref{altproof}, it seems that in the case (T) above the measure $\mu_{X,v}^\text{sph}$ is still singular enough to allow \eqref{competing-conditions} to be satisfied. This of course aligns with the conditions of our theorems, which only require $G_{v_0}$ to be non-split.

\end{document}